\documentclass[12pt,leqno,twoside]{amsart}

\makeatletter
\@namedef{subjclassname@2020}{\textup{2020} Mathematics Subject Classification}
\makeatother

\usepackage[latin1]{inputenc}
\usepackage[T1]{fontenc}
\usepackage[colorlinks=true, pdfstartview=FitV, linkcolor=blue, citecolor=blue, urlcolor=blue]{hyperref}
\usepackage{amstext,amsmath,amscd, bezier,indentfirst,amsthm,amsgen,enumerate, geometry}
\usepackage[all,knot,arc,import,poly]{xy}
\usepackage{amsfonts,color, soul}  
\usepackage{amssymb}
\usepackage{latexsym}
\usepackage{epsfig}
\usepackage{graphicx}
\usepackage{srcltx}
\usepackage{enumitem}
\usepackage{tikz}
\usepackage{wasysym}

\newtheorem{theorem}{Theorem}[section]
\newtheorem{corollary}[theorem]{Corollary}
\newtheorem{lemma}[theorem]{Lemma}
\newtheorem{proposition}[theorem]{Proposition}

\theoremstyle{definition}
\newtheorem{definition}[theorem]{Definition}
\theoremstyle{remark}
\newtheorem{remark}[theorem]{\sc Remark}
\newtheorem{example}[theorem]{\sc Example}

\renewcommand{\Box}{\square}    

\newcommand{\Sing}{{\mathrm{Sing\hspace{2pt}}}}
\newcommand{\Reg}{{\mathrm{Reg\hspace{1pt}}}}

\newcommand{\rank}{{\mathrm{rank\hspace{2pt}}}}

\newcommand{\ord}{{\mathrm{ord}}}

\newcommand{\im}{{\mathrm{Im\hspace{2pt}}}}

\newcommand{\cl}{{\mathrm{closure}}}

\newcommand{\pr}{{\mathrm{pr}}}
\newcommand{\red}{{\mathrm{red}}}

\newcommand{\jac}{{\rank\mathrm{Jac}}}

\newcommand{\thick}{\mathrm{thick}}
\newcommand{\thin}{\mathrm{thin}}

\newcommand{\e}{\varepsilon}
\newcommand{\m}{\setminus}
\newcommand{\s}{\subset}

\newcommand{\fin}{\hspace*{\fill}$\Box$\vspace*{2mm}}

\newcommand{\cA}{{\mathcal A}}
\newcommand{\cB}{{\mathcal B}}

\newcommand{\cE}{{\mathcal E}}

\newcommand{\cI}{{\mathcal I}}

\newcommand{\cO}{{\mathcal O}}
\newcommand{\cP}{{\mathcal P}}

\newcommand{\cU}{{\mathcal U}}

\newcommand{\bC}{{\mathbb C}}
\newcommand{\bE}{{\mathbb E}}
\newcommand{\bP}{{\mathbb P}}
\newcommand{\bN}{{\mathbb N}}
\newcommand{\bZ}{{\mathbb Z}}

\begin{document}

\title[Image of analytic map germs]{The image complexity of an analytic map germ}

\author{\sc Cezar Joi\c ta}
\address{Institute of Mathematics of the Romanian Academy, P.O. Box 1-764,
 014700 Bucharest, Romania and Laboratoire Europ\' een Associ\'e  CNRS Franco-Roumain Math-Mode}
\email{Cezar.Joita@imar.ro}

\author{Mihai Tib\u ar}
\address{Univ. Lille, CNRS, UMR 8524 -- Laboratoire Paul Painlev\'e, F-59000 Lille,
France}  
\email{mihai-marius.tibar@univ-lille.fr}




\thanks{The authors acknowledge support of the project ``Singularities and Applications'' -  CF 132/31.07.2023 funded by the European Union - NextGenerationEU - through Romania's National Recovery and Resilience Plan,  and support by the grant CNRS-INSMI-IEA-329. 
}

\begin{abstract}

The image of a holomorphic map germ  is not necessarily locally open, and it is not always well-defined as a 
 set germ. We find the structure of what becomes the image of a map germ when the target is a surface. We encode it as a decorated tree and we derive a complexity degree.
 
\end{abstract}


\maketitle

\section{Introduction}\label{intro}
  The need of assuming that the image of an analytic map germ is well-defined  \emph{as a set germ} has been pointed out by John N. Mather in his fundamental study of the stratified structure of maps \cite{Ma}. To see what this means, let  $F :(X,\mathfrak{p}) \to (Y,y)$ be a non-constant complex analytic map germ between irreducible complex analytic set germs. 
 Consider an embedding $(X,\mathfrak{p})\subset (\bC^N, 0)$, and the intersection $X\cap B_{\e}$ of some representative of $X$  with a ball $B_{\e}\s \bC^N$ of small enough radius $\e>0$ centred at $0$.
  Then the map germ $F$ has a well defined  \emph{local image} if and only if  the germ at $y$ of the set $F(B_{\e}\cap X)$ is independent of the radius $\e$.  Map germs as simple as  the blow-up $(x,z) \mapsto (x, xz)$ at the origin of $\bC^{2}$  do not have a well defined local image. 
  
  This extends the notion of \emph{locally open maps},  see e.g. \cite{Hu}: the map $F$ germ is locally open at $y$ if for any radius $\e>0$,   the set $F(B_{\e}\cap X)$ contains $y\in Y$ as an interior point. For instance, a  non-constant holomorphic function germ $f : (\bC^{n}, 0) \to (\bC, 0)$ is locally open, as a consequence of the Open Mapping Theorem.
  A criterion for the image of $F$ to be locally open, which responds to Huckleberry's conjecture in \cite{Hu}, has been proved recently, see Remark \ref{r:huck} below.
  
Moreover, in the case $\dim_{y}Y =2$, it has been proved in \cite{JT-arkiv} that if  the image of $F$ is a well defined set germ at $y$,  then this is an \emph{irreducible analytic set germ}, see  Proposition \ref{p:arkivprop} and \S\ref{s:extension}.

\smallskip 
 
      The following natural question occurs: 
      
 \medskip

\noindent
(*) \hspace{.7cm} \emph{If the image of $F: (X,\mathfrak{p}) \to (Y,y)$ is \emph{not well-defined} as a set germ,  then what is it?  If one  can still give a meaning to the ``image of the map germ $F$'',   then what is the structure of this object?}

\medskip

We propose here the following new definition. 
  
  Let $\cA_{X,\mathfrak{p}}$ denote the set of irreducible curve\footnote{We consider here curves as geometric objects,  and not as arcs  defined by given parametrisations.} germs $(C,\mathfrak{p})\subset (X,\mathfrak{p})$. There is a well-defined mapping:
\[ F_{*} : \cA_{X,\mathfrak{p}} \to \cA_{Y,y}, \ \ \ (C,\mathfrak{p})\longmapsto \bigl( F(C),y\bigr) .\]

\begin{definition}[\textbf{\emph{The image of $F$}}] \ \\
 Let $(X,\mathfrak{p})$ be an irreducible germ of an $n$-dimensional complex analytic space, $n\geq 2$ and let $F :(X,\mathfrak{p})\to (Y, y)$ be a 
 holomorphic map germ.
We call  \emph{image of $F$} the subset of curve germs $F_*(\cA_{X,\mathfrak{p}})\subset \cA_{Y,y}$.  
\end{definition}

    Denoting by $0_{y}$  the degenerated curve germ at $y$,  the map germ $F$ is constant if and only if  $F_*(\cA_{X,\mathfrak{p}}) = 0_{y}$.

\begin{remark}\label{r:huck}
The equality $F_*(\cA_{X,\mathfrak{p}}) = \cA_{Y,y}$ is equivalent to ``$F$ has no gap curve''  in the terminology of  
\cite[\S 3.2]{JT-arkiv}. 
Therefore, by \cite[Theorem 3.5]{JT-arkiv}, the assertion ``$F$ is open at $y$'' is equivalent to the equality $F_*(\cA_{X,\mathfrak{p}})=\cA_{Y,y}$.
\end{remark}

 \medskip

    We answer the question (*)  by finding an algorithmic way to describe $F_*(\cA_{X,\mathfrak{p}})$ in case target space $Y$ is a surface\footnote{While the problem remains open for $\dim Y >2$.}.
We will see that the complexity of the image is encoded by a certain new type of decorated finite tree.

To give here a glance of how our approach works, we stick for now to the smooth case  $(Y, y) = (\bC^{2},0)$. 

Let $\tau_1:S_1\to \bC^2$ be the blow-up at the origin, and let $E_1$ be the exceptional divisor.  The set of the proper transforms by $\tau_1$ of all curve germs in $\cI := F_*(\cA_{X,\mathfrak{p}})$ will be denoted  by $\tau_1^*\cI$.
For some $\xi\in E_1$,  let then 
$$(\tau_1^*\cI)_\xi := \bigl\{ \gamma \in \cA_{S_1,\xi}  \mid F_*(\gamma)\in \cI \bigr\}.$$  

We thus have the equality
$\tau_1^*I=\bigcup_{\xi\in E_1}(\tau_1^*\cI)_\xi$.
 A point $\xi\in E_{1}$ can be of one of the following 3 types only (see also the equivalent Definition \ref{d:points}):
 
 \medskip
 
 $\bullet$ \emph{determined-empty},  if $(\tau_1^*\cI)_\xi= 0_{\xi}$.
 
 $\bullet$ \emph{determined-full}, \  if $(\tau_1^*\cI)_\xi=\cA_{S_1,\xi}$.
 
 $\bullet$ \emph{undetermined},  \ \  if  $(\tau_1^*\cI)_\xi\neq 0_{\xi}$ and
$(\tau_1^*\cI)_\xi\neq \cA_{S_1,\xi}$.

 \medskip

These definitions make sense for any $\cI\subset\cA_{\bC^2,0}$, however we will show in Lemma \ref{l:finite} and Lemma \ref{l:onepoint} that if $\cI$ is the image of a holomorphic map germ, then precisely one of the following two situations may occur:

\smallskip
 \textbf{(a)} there are at most finitely many undetermined points in $E_1$ and all other points are determined-full; we then say that  $E_1$ is  \emph{thick},\\
 or
 
 \smallskip
\textbf{(b)} there exists one point $\xi_1\in E_1$ such that any other point $\xi\in E_1\setminus\{\xi_1\}$ is determined-empty; we then say that  $E_1$ is  \emph{thin}.

 \medskip

In the second step we blow-up simultaneously at all undetermined points of $E_1$, say by the map $\tau_{2}$. This map has finitely many exceptional divisors $E_{2,i}$, each of which is thick or thin.  Similarly, in the 3rd step we blow-up simultaneously at all undetermined points of the exceptional divisors created at step 2, and we call $\tau_{3}$ this blow-up map. We continue in this way, step-by-step.  In  \S\ref{s:finiteness} we show that after finitely many steps there are no more undetermined points, and therefore this process ends.  The exceptional divisor $\bE$ of the total blow-up $\tau$ is a tree with thick and thin normal crossing branches.



\subsection{Definition of the \emph{blossom tree} $\cB_{F}$}\label{ss:blossomtree}\ 

Let 
\[
 S=S^{(s)} \stackrel{\tau_{s}}{\longrightarrow} S^{(s-1)} \stackrel{\tau_{s-1}}{\longrightarrow} \cdots \stackrel{\tau_{1}}{\longrightarrow} \bC^{2}
\]
be a sequence of successive blow-ups $\tau_{k}$, $k\in \overline{1,s}$, where $\tau_{k}$, for $k>1$, is the result of blowing up simultaneously all the undetermined  points which occur on the exceptional divisors created at the step $k-1$.  By the finiteness result of  \S\ref{s:finiteness} we may assume that the blow-up $\tau_{s}$ has no more undetermined points.
 
 Let $\tau : (S, E)  \to (\bC^{2},0)$ denote the composition of the finite sequence of blow-ups, and let $\bE := \cup_{i}E_{i}$ be the exceptional divisor, where each irreducible component $E_{i}$ is isomorphic to $\bP^{1}$. At step $k$, the blow-up $\tau_{k}$ at finitely many points yields a number of irreducible exceptional divisors which become thick or thin. They preserve their names, i.e. ``thick'' or ``thin'', until the end of the process.  
 
 For such a component $E_{i}\subset E$, we set:
 $$ \cE^{\thick}_{i} :=  \bigsqcup_{\xi\in E_{i}}\cA_{S, \xi}  \ \ \mbox{ or } \  \  \cE^{\thin}_{i} :=  \bigsqcup_{\xi\in E_{i}}0_{\xi}, $$
 and call $\cE^{\thick}_{i}$ a \emph{thick branch}, and $\cE^{\thin}_{i}$ a \emph{thin branch},  respectively.   We  will also call \emph{terminal} a thin branch which has a determined-full point, called \emph{terminal point}. For a terminal point $\xi_{j}\in E$ we set:
 $$ \cP_{j} := \cA_{S, \xi_{j}}.$$
 
  A terminal branch corresponds therefore to an exceptional divisor $E_{i}$ which occurs at the last step of some sequence of ramified blow-ups. Such a ``last step divisor'' necessarily yields a thin branch (follows from the algorithm \S\ref{ss:algo}, see the footnote at Step 1).
  
  The undetermined points for the non-terminal branches, or the terminal point of a terminal branch, will be called \emph{special points}. The special point of a thin branch created at some  step of the blow-up process occurs precisely at the intersection of the new divisor with the preceding divisor from which it grows.
   To each branch, one attaches a \emph{label} containing the following information: (1).  the place of each of its special points, and (2). the number of the step when it appears in the blow-up process.

 The \emph{blossom tree} $\cB_{F}$ associated with $F$ is then, by definition, the collection of labelled thick and thin branches:  
 $$\cB_{F} := \bigcup_{i}\cE^{\thick}_{i} \cup \bigcup_{j}\cE^{\thin}_{j}$$
except when the blow-up process yields a linear sequence of thin branches (thus without thick branches), in which case there is a unique\footnote{A trivial example is the blow-up map $(x,y)\mapsto (x,xy)$, the blossom tree of which reduces to a thick point.}  terminal ``thick'' point with structure $\cP$. We then set by definition:
 $$\cB_{F} :=  \bigcup_{j}\cE^{\thin}_{j} \cup \cP$$
   
   \bigskip

The map $\tau : S \to \bC^{2}$ restricts to map germs $\tau_{|} : (S,\xi) \to (\bC^{2},0)$, for any $\xi \in E$, which induce maps  ${\tau_{|}}_{*} : \cA_{S, \xi} \to \cA_{\bC^{2},0}$. 

With these notations we may state our main result:  

\begin{theorem}\label{t:main}
 If  $F:(X,\mathfrak{p}) \to (\bC^{2}, 0)$ is a holomorphic map germ the image of which is not well-defined as a set germ, then 
there is a unique blossom tree $\cB_{F}$ such that:
$$F_*(\cA_{X,\mathfrak{p}}) = \tau_{*}(\cB_{F}).$$
\end{theorem}
%
It follows from the construction that  two holomorphic map germs having the same image generate the same blossom tree; in particular, we show that
the blossom tree $\cB_{F}$ may be obtained from a map germ defined on $(\bC^2, 0)$:

\begin{theorem}\label{t:main2}
 Let $(X,\mathfrak{p})$ be a complex irreducible space germ of dimension $n\ge 2$, and let 
$F=(f,g):(X,\mathfrak{p})\to (\bC^2, 0)$ be a holomorphic map germ. Then there exists a holomorphic map germ 
$G:(\bC^2, 0)\to (\bC^2, 0)$ such that $G_*(\cA_{\bC^2,0})=F_*(\cA_{X,\mathfrak{p}})$. In particular, $\cB_{G}$ coincides with $\cB_{F}$.
\end{theorem}

The proof of Theorem \ref{t:main} is developed as an algorithm in \S\ref{s:blowup}. Its cornerstone is the finiteness of this algorithm,  shown in \S\ref{s:finiteness}, with conclusion formulated in \S\ref{ss:proofmain}.

 Our Theorem \ref{t:singsurf} gives the extension of Theorem \ref{t:main} to holomorphic map germs $F:(X,\mathfrak{p}) \to (Y, y)$, where $(Y,y)$ is an irreducible  surface germ.  Our examples in \S\ref{s:ex} show several types of blossom trees,  along with some details of our Algorithm \ref{ss:algo} computations.
 
 \
 
 The above results allow us to introduce a complexity degree for the image, as follows. Let us first recall that for two irreducible curve germs $C_{1},  C_{2} \in \cA_{\bC^2,0}$, one classically says that their \emph{order of tangency}  is $m\ge 0$ if, by blowing up successively, their proper images can be separated  precisely after $m+1$ blow-ups and no less. We denote this by 
 $$\tan(C_{1}, C_{2}) =m.$$
   In particular,  $C_{1},  C_{2}$ have different tangent cones if and only if  $\tan(C_{1}, C_{2}) =0$.

\begin{definition}[\textbf{\emph{Image complexity degree}}] \label{d:degree} \ \\
     For some curve germ $C \in F_*(\cA_{X,\mathfrak{p}})\subset \cA_{\bC^2,0}$, let us set:
$$\kappa_{C}(F) := \min \Bigl\{   m\in\bN \mid   D\in \cA_{\bC^2,0}, \  \tan(D,C) \ge m  \Longrightarrow  D \in F_*(\cA_{X,\mathfrak{p}})  \Bigr\}.
$$
We then define the \emph{image complexity degree} of $F$ as:
$$
 \kappa(F):=\max  \bigl\{ \kappa_{C}(F) \mid   C  \in F_*(\cA_{X,\mathfrak{p}})  \bigr\} .
 $$
\end{definition}
We show by Theorem \ref{t:degree} that 
if the image of $F$ is not well-defined as a set-germ, then $\kappa(F)$ is precisely the maximal length of chains of branches in the blossom tree $\cB_{F}$.




\bigskip

\section{When the local image is a well-defined set germ}\label{ss:welldefinedimage}


 Let $(X,\mathfrak{p})\subset (\bC^N, 0)$ be a complex irreducible space germ of dimension $n\geq 2$, and let $F:(X,\mathfrak{p})\to (\bC^2, 0)$ be a non-constant holomorphic map germ.  
 
%
Let us first discuss the situation when $F$ has a well-defined image as a set germ, and how this is characterized.
 
 \begin{proposition}\label{p:arkivprop}\cite[Proposition 2.1]{JT-arkiv}
If  the image of $F$ is a well defined set germ at the origin,  then either $(\im F, 0)=(\bC^2,0)$ and thus $F_*(\cA_{X,\mathfrak{p}}) =\cA_{\bC^{2},0}$, or $(\im F,0)=(C, 0)$ where  $(C,0)\subset (\bC^2,0)$ is an irreducible  complex curve germ, and in this case $F_*(\cA_{X,\mathfrak{p}})$ is a single element of  $\cA_{\bC^{2},0}$.  
\end{proposition}

\begin{proof}
This has been stated and proved in \cite{JT-arkiv} for $X = \bC^N$,  of course without the interpretation
as $F_*(\cA_{X,\mathfrak{p}})$ that we have just introduced above.
In the case of a subspace $X\subset \bC^N$, after replacing the balls $B_{\e}\subset \bC^N$ centred at the origin by their intersections  $X\cap B_{\e}$, the proof goes word by word as that in \emph{loc.cit.} We may refer the reader to \S\ref{s:extension} for the more general setting where the target $(\bC^2,0)$ is replaced by an irreducible surface germ  $(Y,y)$.
 \end{proof}

 %
 \begin{proposition}\label{p:n-2dim}
Let $F: (X,\mathfrak{p})\to (Y,y)$  be a holomorphic map germ,  where  $(X,\mathfrak{p})$ and $(Y,y)$ are complex irreducible space germs, such that $\dim(X)=n$ and $\dim(Y)=p$, $n\geq p\ge 1$.  If the germ at $x$ of the fibre  $F^{-1}(y)$  has an irreducible component of dimension $n-p$,  then $F$ is locally open at $y$.
 \end{proposition}
\begin{proof}
Let $B\subset X$ be a small enough open neighbourhood of $\mathfrak{p}$ where the holomorphic map $F$  is defined. Let $S$ be an irreducible
component of the germ $(F^{-1}(y),\mathfrak{p})$, with $\dim_{\mathfrak{p}} S = n-p$.  By repeatedly slicing with  generic hyperplanes at $\mathfrak{p}$, one gets a closed irreducible analytic subset $Z\subset B$ of dimension $p$ such that $\mathfrak{p}$ is  an isolated point of the intersection $Z\cap F^{-1}(y)$. By \cite[Proposition, page 63]{GR}, it then follows that there exist
 an open neighbourhood  $U$ of $\mathfrak{p}$ in $Z$ and an open neighbourhood $V$ of $y$ in $Y$, 
such that $F(U)\subset V$ and that the induced map $F_{|U}:U\to V$ is finite. 
 By the Open Mapping Theorem  (cf \cite[ page 107]{GR}), this implies 
that $F(U)$ is open. It follows that $F(B)$ is open, and since $B$ may be arbitrarily small, we deduce that $F$ is locally open at $y$.
\end{proof}

\begin{proposition}\label{p:curve}
Let $(X,\mathfrak{p})$ be an irreducible complex analytic set germ and let $F=(f,g):(X,\mathfrak{p})\to (\bC^2,0)$ be a nonconstant holomorphic map germ.
Then the image of  $F$ is a curve germ if and only if  $\jac (f,g)<2$ on $\Reg X$.
\end{proposition}
\begin{proof}
Since the implication ``$\Rightarrow$'' is obvious, we will focus on the proof of ``$\Leftarrow$''.

As before, we consider a suitable open neighbourhood $B\subset X$ of $\mathfrak{p}$ where the holomorphic map $F$  is defined, and some representatives of $X$ and of its regular part $\Reg X$. 
If $\rank(f,g)\equiv 0$, 
it follows that $(f,g)$ is constant on $\Reg X$ which is connected, and hence it is constant on $X$, contrary to our assumption. We deduce that 
$\rank(f,g)\equiv 1$ on an open dense subset $\Lambda \subset \Reg X$. 
Then the rank theorem  implies that $\dim_x(f,g)^{-1}\bigl(f(x),g(x)\bigr)=n -1$, for any $x\in \Lambda$.
 
Actually one has the inequality  $\dim_x(f,g)^{-1}(f(x),g(x))\geq n -1$ for any $x\in X$, due to the following result by Remmert:
 \begin{lemma}\cite[Satz 16]{Re0}
Let $F:X\to Y$ be a holomorphic map between complex spaces. Then any $a\in X$ has 
a neighbourhood $U_{a}\subset X$ such that, for any $x\in U_{a}$, one has:
$$ \dim_x F^{-1}(F(x))\leq \dim_a F^{-1}(F(a)).$$
  \fin
\end{lemma}
 However, one cannot have  $\dim_a(f,g)^{-1}(f(x),g(x)) > n-1$ for some $a\in X$, since this would contradict the fact that $X$ is irreducible of dimension $n$, and that $(f,g)$ is not constant. Consequently, the equality $\dim_x(f,g)^{-1}(f(x),g(x))= n -1$ holds  for any $x\in X$ and not only for $x\in \Reg X$.
 
To complete our proof, we only need the following classical result due to Remmert. 

\begin{lemma}\label{l:nara}\cite[Satz 18]{Re0}
Let $X$ and $Y$ be complex spaces such that $X$ is  pure dimensional, and let $F:X\to Y$ be a 
holomorphic map. If $r = \dim_{x} F^{-1}(F(x))$ is independent of $x\in X$, then any point $a\in X$
has a fundamental system of neighbourhoods $\{U_{i}\}$ such that $f(U_{i})$ is analytic at $F(a)$, of dimension $\dim X - r$.
\fin 
\end{lemma}
We may apply Lemma \ref{l:nara} to our setting since the dimension $r= n-1$ is independent of $x\in X$. This concludes the proof of Proposition \ref{p:curve}. 
\end{proof}

\medskip



 \section{The structure of  $F_*(\cA_{X,\mathfrak{p}})$} \label{s:blowup}
 
 From now on we assume that the non-constant holomorphic map germ $F:(X,\mathfrak{p})\to (\bC^2, 0)$ does not have a well-defined image as a set germ.
Let  $F^{-1}(0) := \bigcup_{j=1}^{k}H_{j}$ be the decomposition into irreducible components, understood as space germs with reduced structure. It follows by Proposition \ref{p:n-2dim} that   $\dim H_{j} = n-1$, where $n=\dim X$, thus all $H_{j}$ are  \emph{divisors} as germs at $\mathfrak{p}\in X$.

\medskip
 
We may assume without loss of generality that $\Sing(X)$ does not contain any irreducible component of the central fibre $F^{-1}(0)$.  This condition is achieved when replacing $X$ by its normalisation. Indeed, composing $F$  with the normalisation map does not change the image $F_*(\cA_{X,\mathfrak{p}})$ since normalisation maps are locally open.
The following equivalent condition will be needed in Lemma \ref{l:biholomorphism} and Definition \ref{d:order}:
 \begin{equation}\label{eq:dim-sing}
 \dim\bigl(\Sing(X)\cap F^{-1}(0)\bigr) \leq n-2.
 \end{equation}
 
 \medskip


 \subsection{The blow-up construction}\label{ss:blowup}
 
 Let $\tau_1:\widetilde\bC^2\to \bC^2$ be the blow-up of $0\in \bC^2$, with its exceptional divisor denoted by $E_{1}$. Let 
 $W:= X\times_{\bC^{2}} \widetilde\bC^2$ be the fibered product, endowed with the projections $\pr_1:W\to X$, and  $\pr_2:W\to \widetilde\bC^2$. 
 
 The restriction $\pr_{1|}: \pr_1^{-1}\bigl(X\setminus F^{-1}(0)\bigr) \to X\setminus F^{-1}(0)$ is a biholomorphism by construction. Let then $X_1$ denote the closure of $\pr_1^{-1}\bigl(X\setminus F^{-1}(0)\bigr)$
 in $W$. By definition,  $X_1$ is an irreducible component of $W$, while the other irreducible components are $H_j\times E_{1}$, for all $j=1, \ldots k$.
 Denoting by $\pi_1$ the restriction of $\pr_1$ to $X_1$,  and by $F_1$ the restriction of $\pr_2$ to $X_1$, 
  we obtain the following commutative diagram:
 \begin{equation}\label{diagram}
\xymatrix{
    X_1  \ar[d]_{\pi_1} \ar[rr]^{F_1}    &    &    \widetilde\bC^2\ar[d]_{\tau_1}  \\
         X \ar[rr]^F   &   &  \bC^2 ,
} 
\end{equation}
where $\pi_1$ is proper and surjective.
 \medskip
 
 Let $\Sigma:=\Sing X \bigcup\Sing \{f =0\}^{\red} \bigcup\Sing \{g =0\}^{\red}$, where the notation 
$Z^{\red}$ means the reduced structure of the analytic space $Z$. By definition $\Sigma$ is an analytic subset of $X$ and does not contains any irreducible component of $F^{-1}(0)$. 
 
\begin{lemma} \label{l:biholomorphism}
The restriction 
$${\pi_1}_{|}: \pi_{1}^{-1}\biggl(X\setminus \bigl(F^{-1}(0)\cap \Sigma\bigr)\biggr) \to X\setminus \bigl(F^{-1}(0)\cap \Sigma\bigr)$$
is a biholomorphism.
\end{lemma}
\begin{proof}
The map  ${\pi_1}_{|}: \pi_{1}^{-1}\bigl(X\setminus F^{-1}(0)\bigr) \to X\setminus F^{-1}(0)$
is by definition a biholomorphism.
What we need to prove in addition is that for any $j\in\{1,\dots,k\}$, and for any fixed point $x\in H_j\setminus \Sigma$, there exists a neighbourhood
 $U_{x}$ of $x$  in $X$ such that $\pi_1:\pi_1^{-1}(U_{x})\to U_{x}$ is a biholomorphism.  
 
 Let $U'\subset \Reg X\setminus  \Sigma$ be a connected open neighbourhood of our fixed point $x$, such that $H_j \cap U'=\{\psi=0\}$  for some holomorphic function $\psi\in\cO(U')$, and that 
  $f=\psi^p\tilde f$, $g=\psi^q\tilde g$ for some integers $p, q>0$,  where $\tilde f$ and $\tilde g$ have no zeroes on $U'$. We have the following local presentation:
$$U'\times_{\bC^{2}} \widetilde\bC^2 = \biggl\{\bigl(x,(u,v),[\zeta_0:\zeta_1]\bigr)\in X\times\bC^2\times\bP^1 \ \vert \ (f,g)(x)=(u,v),\ u\zeta_1=v\zeta_0\biggr\}.$$
In case $q\leq p$, we then get: 
$$ \pi_1^{-1}(U') = \biggl\{\bigl(x,(u,v),[\zeta_0:\zeta_1]\bigr)\in X\times\bC^2\times\bP^1  \ \vert \  (f,g)(x)=(u,v),\ 
\psi^{p-q}(x)\tilde f(x)\zeta_1=
\tilde g(x)\zeta_0\biggr\}.$$
This shows that the mapping:
\begin{equation}\label{rel:bih}
U'\ni x\longmapsto \biggl(x,\bigl(f(x),g(x)\bigr), \bigl[ \psi^{p-q}(x)\tilde f(x):\tilde g(x)\bigr]\biggr)
\end{equation}
is the holomorphic inverse of $\pi_1$. Thus our claim holds for $U_{x} = U'$. 

The case $p\leq q$ is completely similar.
\end{proof}

 Lemma \ref{l:biholomorphism} has the following immediate consequence:
\begin{corollary}\label{c:not-sing}
 For any $j= 1\ldots , k$, the inverse image $\pi_1^{-1}(H_j)$ has a unique irreducible component $\widetilde H_j$ such that $\pi_1(\widetilde H_j)=H_j$,  and that $\widetilde H_j\not\subset\Sing(X_1)$.  
 
 We will say that $\widetilde H_j$ is the \emph{proper transform of $H_j$ by $\pi_1$}.
 \fin
\end{corollary}
Let us set the notation:
 $$ \cA_{E_{1}} := \bigsqcup_{\xi\in E_{1}}\cA_{\widetilde \bC^2,\xi}.$$
 
 \smallskip
 
The next lemma shows that our desired image $F_*(\cA_{X,\mathfrak{p}})$ of $F$  is the push-forward by $\tau_{1}$ of the subset 
${F_1}_*\bigl(\bigcup_{z\in \pi_1^{-1}(p)}\cA_{X_1,z}\bigr)$ of $\cA_{E_{1}}$.

\begin{lemma}\label{l:equal}
The following equality holds:
 $$(\tau_1\circ F_1)_*\biggl(\bigcup_{z\in \pi_1^{-1}(\mathfrak{p})}\cA_{X_1,z}\biggr)=F_*(\cA_{X,\mathfrak{p}}).$$
\end{lemma}
\begin{proof}
\noindent  ``$\subset$'' follows from the commutativity of the diagram (\ref{diagram}). \\
 \noindent  ``$\supset$''. Working in the notations of curve germs, let $(\gamma, 0)\in F_*(\cA_{X,\mathfrak{p}})$, with $(\gamma, 0) \not = 0_{0}$,  and let
$(C,p)\in \cA_{X,\mathfrak{p}}$ such that $F(C)= \gamma$. 
Let $(\tilde \gamma, \xi)$ denote the proper transform of $\gamma$ by the blow-up $\tau_{1}$, where $\xi$ is a point of the exceptional divisor $E_{1}$. 

Then the fibered product $C\times_{\gamma}\tilde \gamma$ is a curve germ at $z:= (\mathfrak{p}, \xi)$ and therefore it is an element of  $\cA_{X_1,z}$, whose image by $\tau_1\circ F_1$ is precisely the curve germ $(\gamma,0)$.
\end{proof}

\medskip

Since  $X_1\subset X\times_{\bC^{2}} \widetilde\bC^2$ is an irreducible analytic subset, the intersection  
$X_1\cap(\{\mathfrak{p} \}\times E_{1}) = \pi_1^{-1}(\mathfrak{p})$ consists of finitely many points, or  it is of dimension 1 and therefore
$\pi_1^{-1}(\mathfrak{p}) = \{\mathfrak{p}\}\times E_{1}$. 

\smallskip


In order to distinguish between the images  ${F_1}_*\bigl(\cA_{X_1,(\mathfrak{p},\xi)}\bigr) \subset \cA_{\widetilde \bC^2,\xi}\subset \cA_{E_{1}}$  at the points $\xi$ along $E_{1}$, we introduce the following definition, which is equivalent to that given in  \S\ref{intro}:

\begin{definition}\label{d:points}
 A point $\xi\in E_{1}$ will be called:
 
 $\bullet$ \emph{determined-empty}, \  if $\xi\not\in F_1(\pi_1^{-1}(\mathfrak{p}))$.
 
 $\bullet$ \emph{determined-full}, \ \ \ \ if $\xi\in F_1(\pi_1^{-1}(\mathfrak{p}))$ and
${F_1}_*\bigl(\cA_{X_1,(\mathfrak{p},\xi)}\bigr) = \cA_{\widetilde \bC^2,\xi}$.

 $\bullet$ \emph{undetermined},  \ \ \ \ \ \ if  $\xi\in F_1(\pi_1^{-1}(\mathfrak{p}))$ and
${F_1}_*\bigl(\cA_{X_1,(\mathfrak{p},\xi)}\bigr) \subsetneq \cA_{\widetilde \bC^2,\xi}$.
\end{definition}

\medskip

We have denoted by  $k>0$ the number of  irreducible components of $F^{-1}(0)$, that we have assumed to be all of dimension $(n-1)$. We then have:
\begin{lemma}\label{l:finite}
There are at most $k$ undetermined points on $E_{1}$. 
\end{lemma}
\begin{proof}
 If $\xi\in E_{1}$ is an undetermined point then
$F_1:(X_1,(\mathfrak{p},\xi))\to (\widetilde\bC^2,\xi)$ is not open at $\xi$,  and therefore by Proposition \ref{p:n-2dim}, all irreducible components of  $F_1^{-1}(\xi)$ have dimension $n-1$. 

By definition, we have $F_1^{-1}(\xi)\subset \pr_2^{-1}(\xi)$, and $\pr_2^{-1}(\xi)=\bigl\{(x,\xi) \ \vert \ x\in X,\ F(x)=\tau_1(\xi)=0\bigr\}$, thus:
 $$ F_1^{-1}(\xi)\subset \pr_2^{-1}(\xi)= F^{-1}(0)\times\{\xi\}=\bigcup_{j=1}^k H_j \times\{\xi\}.$$ 
We deduce that there exists a subset $I_\xi\subset\{1,\dots,k\}$ such that $F_1^{-1}(\xi)= \bigcup_{j=1}^k H_j \times\{\xi\}$, and in particular $H_j \times\{\xi\} \subset X_1$ for any $j\in I_\xi$.  Combining with Corollary \ref{c:not-sing} and its notations, we get the equality $H_j\times\{\xi\} = \widetilde H_{j}$, for any $j\in I_\xi$, where $\widetilde H_{j}\subset X_{1}$ is the unique component of $\pi_1^{-1}(H_j)$ which projects onto $H_{j}$. This equality also tells that if $\xi \not= \xi'$ are two undetermined points, then 
$I_\xi$ and $I_{\xi'}$ are disjoint sets. Finally, since the disjoint union $\bigsqcup_{\xi} I_\xi$,  over all undetermined points $\xi\in  E_{1}$,  is a subset of $\{1, \ldots, k\}$, our claim follows.
\end{proof}

\begin{lemma}\label{l:onepoint}
If $\dim \pi_1^{-1}(\mathfrak{p}) =0$ then $\pi_1^{-1}(\mathfrak{p})$ is precisely one single point.
\end{lemma}
\begin{proof}
First, let us observe that $\pi_1^{-1}(\mathfrak{p})$ is non-empty since $F$ is a non-constant map. Next, let us assume by contradiction that $\pi_1^{-1}(\mathfrak{p})=\{(\mathfrak{p},\xi_1),\dots,(\mathfrak{p},\xi_m)\}$ with $m\geq 2$. We choose $V_1,\dots,V_m\subset X_1$ pairwise disjoint open neighbourhoods of these points,  and a connected open neighbourhood $V\subset X$ of $\mathfrak{p}$ such that $\pi_1^{-1}(V)\subset \bigsqcup_{j=1}^mV_j$.
The properness of $\pi_1$ insures the existence of such a neighbourhood $V$.  Since $V_{i}$ contains the point $(\mathfrak{p},\xi_i)$, it follows that $\pi_1^{-1}(V\setminus F^{-1}(0))\cap V_j\neq\emptyset$ for $j=1,\ldots ,m$. On the other hand, as $(X,\mathfrak{p})$ is irreducible, the complement 
$V\setminus F^{-1}(0)$ is also connected. Since the restriction ${\pi_1}_{|}: \pi_1^{-1}(V\setminus F^{-1}(0)) \to V\setminus F^{-1}(0)$ is an isomorphism (by the definition of the square diagram \eqref{diagram}), it follows that $\pi_1^{-1}(V\setminus F^{-1}(0))$ is also connected. This contradicts the fact that it is covered by the disjoint union of open sets $\bigsqcup_{j=1}^mV_j$, $m\ge 2$.
\end{proof}

By the above results, we get the following definition that has been announced in \S\ref{intro}:
\begin{definition}[\emph{\textbf{Thick and thin branches}}]\label{r:concl}  \
\begin{itemize}
 \item[(i).]  If $\dim \pi_1^{-1}(\mathfrak{p}) =0$,  then all the points of $E_{1}$ are determined-empty except of \emph{one single point} (cf Lemma \ref{l:onepoint}),  which is either  undetermined or determined-full. 
 
 We say that $E_{1}$ yields a \emph{thin branch} $\cE^{\thin}_{1} :=  \bigsqcup_{\xi\in E_{1}}0_{\xi}$.
 
 \medskip
 
\item[(ii).]  If $\pi_1^{-1}(\mathfrak{p}) = \{\mathfrak{p}\}\times E_{1}$, then all the points of $E_{1}$ are determined-full except of finitely many, which are undetermined (cf Lemma \ref{l:finite}).  There is \emph{at least one} such undetermined point;  if none, then the image of $F$ is open at $0\in \bC^{2}$.  

We say that $E_{1}$ yields a \emph{thick branch} $\cE^{\thick}_{1} :=  \bigsqcup_{\xi\in E_{1}}\cA_{S, \xi}$.
\end{itemize}
\end{definition}

\medskip

\subsection{The algorithm}\label{ss:algo}

Let $F : (X,\mathfrak{p}) \to (\bC^{2}, 0)$ be a non-constant map germ such that its image is not a well-defined set germ. 

\medskip

\noindent \textbf{Step 1}.
We apply to $F$ the above blow-up construction.  We will use the following new notations: $S_0 :=\bC^2$, $S_1:=\widetilde \bC^2$.
The following situations occur:

\smallskip

 \noindent $\bullet$  If the exceptional divisor $E_{1}$ has no undetermined point, then  $E_{1}$ yields a thin branch with a single determined-full point $\xi$, and all the other points are determined-empty. The algorithm stops at this branch, that we will call \emph{a terminal branch}\footnote{It follows that all terminal branches are thin.}.
The image $F_*(\cA_{X,\mathfrak{p}})$ is completely determined by the push-down by $\tau_{1}$ of this thin branch, since by Lemma \ref{l:equal}:
$$F_*(\cA_{X,\mathfrak{p}}) = {\tau_{1}}_{*}(\cA_{\widetilde \bC^{2},\xi}).$$ 

\medskip

 \noindent $\bullet$  If the exceptional divisor $E_{1}$ has a finite set of undetermined points (cf Lemma \ref{l:finite}), let it be $\{\xi_1,\dots, \xi_\ell\}$, $\ell \ge 1$,  then either $E_{1}$ supports a thick branch, or $E_{1}$ supports a thin branch and in this case $\ell =1$. 

We then apply the blow-up construction (i.e. diagram \eqref{diagram}) simultaneously to 
 the following restrictions  of $F_{1}$ as map germs:
 
 $${F_{1}}_{|} : (X_{1}, (\mathfrak{p},\xi_r)) \to (S_{1}, \xi_{r}), \ \ \mbox{ for } r= 1, \ldots, \ell . $$
 
 New exceptional divisors $E_{2,1}, \ldots , E_{2,\ell}$ are created by these blow-up constructions, one 
 for each point $\xi_{1}, \ldots ,\xi_{\ell}$.
 
 \
 
 \noindent \textbf{Step 2.} We are now examining $E_{2,r}$, for $r= 1, \ldots, \ell$. Only two cases may occur:
 
 \smallskip
 \noindent $\bullet$     If $E_{2,r}$ has no undetermined points,  then the algorithm stops at this thin branch as a terminal branch, like in Step 1 above.
 Remark that the single determined-full point $\xi$ of this thin branch occurs precisely at the intersection point of $E_{2,r}$ with the proper image of $E_{1}$ by the blow-up $\tau_{2,r}$ of the smooth surface germ $(S_{1}, \xi_{r})$.
 
 \smallskip
   
 \noindent $\bullet$   If $E_{2,r}$ has undetermined points, then we continue as in Step 1 above   with new blow-up constructions at those undermined points.

\
 
The result of the algorithm is a connected blossom tree $\cB_{F}$ with thick and thin branches, where the ramifications occur at the steps where the branches
 have more than one undetermined points.  We shall prove in the next section that this tree $\cB_{F}$ is finite.
 
One then recovers the image $F_*(\cA_{X,\mathfrak{p}})$ by pushing down the branches by the blow-up maps $\tau_{m}$ in the inverse order of the creation steps. To know in which order to do this, we may use either:
 
 $(1).$ the picture of $\cB_{F}$ as a tree with the two distinct types of branches,   each one labelled by  the step rank  at which it is created, 
 
 or 
 
 $(2).$ the dual diagram,  where the thin branches are represented by the symbol  $\odot$
 and the thick branches by the symbol $\CIRCLE$,
and each one labelled by  the  step rank at which it is created.   From the divisor created at the first step (which can be a thin branch  or a thick branch) leave several arrows (namely a number of $\# I_{\xi} \le k$ arrows) pointing to the branches created at step 2,  which are thick or thin; and so on.
 
In our examples  we have chosen the dual diagram representation.

\section{Proof of the finiteness of the algorithm}\label{s:finiteness}

We will prove here  by \emph{reductio ad absurdum} that the algorithm is finite, in the hypothesis that $F = (f,g)$ does not have a well-defined image as a set germ. \\
Assuming the  contrary means that the tree $\cB_{F}$  contains some infinite chain of  branches, which corresponds to a chain
of diagrams as follows:

\begin{equation}\label{diagram2}
\xymatrix{
 : \ar[d] & & : \ar[d] \\
\hspace{-0.8cm} \mathfrak{p}_{m+1}\in  X_{m+1}  \ar[d]_{\pi_{m+1}} \ar[rr]^{F_{m+1}}   &    &    S_{m+1}\ar[d]_{\tau_{m+1}}\ni\xi_{m+1} \stackrel{\chi_{m+1}}{\longmapsto} 0 \in \bC^{2}\hspace{-3.75cm}  \\
\hspace{-0.8cm} \mathfrak{p}_m\in X_m  \ar[d]_{\pi_{m}} \ar[rr]^{F_m}    &    &   S_m \ar[d]_{\tau_m}\ni\xi_m \ \stackrel{\chi_{m}}{\longmapsto} \ 0 \in \bC^{2}\hspace{-3.50cm}  \\
: & & : } 
\end{equation}
where  $\mathfrak{p}_m= (\mathfrak{p}_{m-1}, \xi_{m})$ and $\xi_m$ is one of the undetermined points 
$\xi_{m,s}$, for $s\in \{1, \ldots , \ell_{m}\}$, and where
 $\tau_{m+1}$ denotes the blow-up of $S_m$ at $\xi_m$.

\

Since $\xi_m$ is undetermined, by definition the image of $F_m$ is not open at $\xi_m$,  and therefore all irreducible components 
of $F_m^{-1}(\xi_m)$ have dimension $n-1$ (cf Proposition \ref{p:n-2dim}). 

We have seen in the proof of Lemma \ref{l:finite} that the projection $\pi_m$ establishes an injective correspondence 
between the set of irreducible components of $F_m^{-1}(\xi_m)$ and the set of irreducible components of 
 $F_{m-1}^{-1}(\xi_{m-1})$ and so, step-by-step downwards by the  by the successive projections $\pi_{m-i}$ in diagram \eqref{diagram2}, an injective correspondence
 between the set of irreducible components of $F_m^{-1}(\xi_m)$ and the set of irreducible components of $F^{-1}(0) = \cup_{j=1}^{k}H_j$.  
 
 With these notations, let us define:
 $$I^{(m)}:=\bigl\{ j\in \{1,\ldots, k\} \  \vert \  H_j \text{ is the image of some component of }F_m^{-1}(\xi_m)\bigr\}.$$
We therefore have a decreasing sequence of non-empty subsets of indices:
$$I^{(m)}\subset I^{(m-1)}\subset  \cdots \subset I^{(0)}:= \{1,\ldots, k\}$$ 
which stabilises, i.e.,  there exists $m_{0}\ge 0$ such that $I^{(m)} = I^{(m_{0})}$ for all $m\ge m_{0}$.

 For some fixed $j_0\in\bigcap_{m\geq 1} I^{(m)}$, let us then denote by  $H_{j_0}^{(m)}$  the unique 
irreducible component of $F_m^{-1}(\xi_m)$ which projects onto $H_{j_0} = H_{j_0}^{(0)}$ by the sequence of $\pi_{j}$ in diagram \eqref{diagram2}.   Modulo the identifications by the projections  $\pi_{j}$, we then have:
\begin{equation}\label{eq:stab}
 H_{j_0}\simeq H_{j_0}^{(1)} \simeq \cdots \simeq H_{j_0}^{(m)} \simeq \cdots.
\end{equation}

\subsection{Local orders along divisors}\label{ss:orders}

We choose a local coordinate system at $\xi_m$, i.e. on an open set $V\subset S_m$ containing $\xi_m$, such that $\xi_{m}$
becomes the origin, and thus  $(S_{m}, \xi_{m})$ identifies with $(\bC^{2}, 0)$ through a chart map (which is denoted by $\chi_{m}$ in diagram \eqref{diagram2}, see also \eqref{eq:localchart} for $\chi_{m+1}$).
In this local coordinate system  one has  $F_{m} = (f_{m}, g_{m})$, and our divisor $H_{j_0}^{(m)}$ is an irreducible component of $(f_{m}, g_{m})^{-1}(0)$. 

\begin{definition}\label{d:order}
 Let $x\in H_{j_0}^{(m)}\m \bigl( \Sing X_m \cup \Sing H_{j_0}^{(m)}\bigr)$, and remark that   $H_{j_0}^{(m)}\m \bigl( \Sing X_m \cup \Sing H_{j_0}^{(m)}\bigr)$ is not empty by \eqref{eq:dim-sing}.
Let $\psi$ be a holomorphic function which defines  $H_{j_0}^{(m)}$ in some neighbourhood $U\subset X_{m}$ of $x$.
One then writes $f_{m}=\psi^p\tilde f_{m}$, $g_{m}=\psi^q\tilde g_{m}$, with $\tilde f_{m}\vert  H_{j_0}^{(m)} \not\equiv 0$ and $\tilde g_{m}\vert  H_{j_0}^{(m)} \not\equiv 0$.  We say that $p:=\ord_{H_{j_0}^{(m)}}f_{m}$,  and $q := \ord_{H_{j_0}^{(m)}}g_{m}$.
Let us also set:
 $$\ord_{H_{j_0}^{(m)}}F:=\min\{p,q\}.$$
\end{definition}
By standard arguments,  $\ord_{H_{j_0}^{(m)}}F$  does not depend on the choices, and it is invariant under local biholomorphisms, in particular independent of the charts.

\begin{lemma}\label{l:order}
$\ord_{H_{j_0}^{(m+1)}}F_{m+1}\leq \ord_{H_{j_0}^{(m)}}F_m$.
\end{lemma}
\begin{proof}
For $U$ as in Definition \ref{d:order}, let $U'\subset U$ be an open set such that $U'\m \bigl( \Sing X_m \cup \Sing H_{j_0}^{(m)}\bigr) \neq\emptyset$,  and that $\tilde f(x)\neq 0$ and $\tilde g(x)\neq 0$ for 
all $x\in U'$. It follows from Lemma \ref{l:biholomorphism}  that the restriction $\pi_{m+1} \vert  \pi_{m+1}^{-1}(U')$ is a biholomorphism onto  $U'$. 

Without loss of generality, we may suppose that $\ord_{H_{j_0}^{(m)}}g_{m}= q\leq p =\ord_{H_{j_0}^{(m)}}f_{m}$. Therefore $\ord_{H_{j_0}^{(m)}}F_m=q$.

Let $V\subset S_m$ be the open set containing $\xi_m$  defined in the beginning of  \S \ref{ss:orders}.
We have that $\tau_{m+1}^{-1}(V)$ is the blow-up of $V\subset S_m$ at $0$, thus it can be written as 
$$\tau_{m+1}^{-1}(V)=\biggl\{\bigl((u,v),[\lambda_0:\lambda_1]\bigr)\in V\times\bP^1 \ \vert \ u\lambda_1=v\lambda_0\biggr\}.$$

For  $x\in U'$ we have:  
$$(F_{m+1}\circ \pi_{m+1}^{-1})(x)=\biggl( \bigl(f_m(x),g_m(x)\bigr), \bigl[\psi_m^{p-q}(x)\tilde f_m(x):\tilde g_m(x)\bigr]\biggr) .$$ 
Since by our assumption,  $F_{m+1} \vert H_{j_0}^{(m+1)} \equiv \xi_{m+1}$, and $H_{j_0}^{(m+1)}$ is the proper transform of $H_{j_0}^{(m)}$ by $\pi_{m+1}$, we get the identity
$\bigl[ \psi_m^{p-q}(x)\tilde f_m(x):\tilde g_m(x)\bigr] \equiv [\alpha_0 :\alpha_1 ]$  on $U'\cap H_{j_0}^{(m)}$, where 
$\xi_{m+1}= \bigl( (0,0),[\alpha_0:\alpha_1]\bigr)$, and $\alpha_1\neq 0$
since $\tilde g_m$ is not identically zero on $H_{j_0}^{(m)}$. 

We will
use the following local chart at  the point $\xi_{m+1}= \bigl( (0,0),[\frac{\alpha_0}{\alpha_1}:1]\bigr)$:
\begin{equation}\label{eq:localchart}
 \tau_{m+1}^{-1}(V)\ni \biggl( (u,v),[\lambda_0:\lambda_1]\biggr) 
\stackrel{\chi_{m+1}}{\longmapsto} \biggl(\frac{\lambda_0}{\lambda_1} -\frac{\alpha_0}{\alpha_1},v\biggr) \in \bC^{2},
\end{equation}
which identifies $\xi_{m+1}$ to the origin $(0,0)\in \bC^{2}$.

In this chart we then get:
\begin{equation}\label{eq:F_{m+1}}
  (F_{m+1}\circ \pi_{m+1}^{-1})(x)=\left(\frac{\psi_m^{p-q}(x)\tilde f_m(x)-(\alpha_0/\alpha_1)\tilde g_m}{\tilde g_m},\psi^q \tilde g_m\right),
\end{equation}
and we conclude  that $\ord_{H_{j_0}^{(m)}}(F_{m+1}\circ \pi_{m+1}^{-1})\leq q$. Therefore we have:
$$\ord_{H_{j_0}^{(m+1)}}F_{m+1}\leq q = \ord_{H_{j_0}^{(m)}}F_m,$$ 
which ends the proof of our lemma.
\end{proof}


\subsection{Stabilisation of the order}\label{ss:orderstab}

Lemma \ref{l:order} implies that there exists a rank $m_0 >0$ after which the order stabilises, namely:
$$\ord_{H_{j_0}^{(m)}} F_m=\ord_{H_{j_0}^{(m_0)}}F_{m_0}, \ \forall m\geq m_0.$$ 

From now on we will fix the minimal  integer $m_0$ with the above stability property, and we write the above results for $m=m_0$.
Let us use the chart $\chi_{m_0}$ to write  $\chi_{m_0}\circ F_{m_0} = (f_{m_0},g_{m_0})$. We may assume, without loss of generality,
that  $\ord_{H_{j_0}^{(m_0)}}(g_{m_0})\leq \ord_{H_{j_0}^{(m_0)}}(f_{m_0})$;
we therefore also have  $\ord_{H_{j_0}^{(m_0)}}(F_{m_0})=\ord_{H_{j_0}^{(m_0)}}(g_{m_0})$.

\medskip
By using the local charts $\chi_{m}$ defined at \eqref{eq:localchart}, let us set, for any $m\ge m_{0}$:
\begin{equation}\label{eq:setU}
\cU_{m} :=  \bigl(X_{m}\setminus g_{m}^{-1}(0)\bigr) \bigcup \Bigl(\Reg X_{m}\setminus \overline{g_{m}^{-1}(0)\setminus H_{j_0}^{(m)}}\,\Bigr).
\end{equation}

Under these notations, we state the following refinement of Lemma \ref{l:biholomorphism}:
 \begin{lemma}\label{l:biholomorphism-2}
If $\ord_{H_{j_0}^{(m_0)}}(g_{m_0})\leq \ord_{H_{j_0}^{(m_0)}}(f_{m_0})$, then the restriction
$${\pi_{m_0+1}}_{\vert} : \pi_{m_0+1}^{-1}(\cU_{m_{0}}) \to \cU_{m_{0}}$$
is a biholomorphism.
\end{lemma}
\begin{proof}
The proof of Lemma  \ref{l:biholomorphism} holds for $F_{m_{0}}$ instead of $F$, and shows that the map $\pi_{m_0+1}$ is a biholomorphism on the set:
\begin{equation}\label{eq:bihol3}
 \pi_{m_0+1}^{-1}\biggl(  \bigl(X_{m_0}\setminus g_{m_0}^{-1}(0)\bigr) \bigcup \Bigl(\Reg X_{m_0} \setminus\bigl(\overline{f_{m_0}^{-1}(0)\setminus H_{j_0}^{(m_0)}}\cup \overline{g_{m_0}^{-1}(0)\setminus H_{j_0}^{(m_0)}}\bigr)\Bigr) \biggr) .
\end{equation}

 Since
$\ord_{H_{j_0}^{(m_0)}}(g_{m_0})\leq \ord_{H_{j_0}^{(m_0)}}(f_{m_0})$, we are indeed in the situation ``$q\leq p$'' 
considered in the proof of Lemma  \ref{l:biholomorphism}. Comparing the set in \eqref{eq:bihol3} with the set $\cU_{m_{0}}$ defined in \eqref{eq:setU},
by following the proof of Lemma  \ref{l:biholomorphism} we notice  that the map \eqref{rel:bih} is well-defined and holomorphic if $\tilde g_{m_0}$ has no zeroes on some small enough neighbourhood $U_{x}$ of any fixed point $x$ of the open set  $\Reg X_{m_0}\setminus \overline{g_{m_0}^{-1}(0)\setminus H_{j_0}^{(m_0)}}$, which is just the property of this set by definition.
\end{proof}

We will tacitly use the chart $\chi_{m_0}$ defined on an open neighbourhood $V\subset S_{m_{0}}$ of $\xi_{m_{0}}$, as introduced in the beginning of \S\ref{ss:orders}. Let us set:
$$\Omega :=\biggl(\bigl(X_{m_0}\setminus g_{m_0}^{-1}(0)\bigr)\bigcup \Bigl( \Reg X_{(m_0)}\setminus \overline{g_{m_0}^{-1}(0)\setminus H_{j_0}^{(m_0)}}\Bigr)\biggr)
\bigcap F_{m_0}^{-1}(V),$$
The open set $\Omega$ is included in the set  $\cU_{m_{0}}$ defined  in Lemma \ref{l:biholomorphism-2}, hence
$\pi_{m_0+1}$ is biholomorphic on $\pi_{m_0+1}^{-1}(\Omega)$.

\begin{lemma}\label{l:level1}
Let  $\xi_{m_{0}+1}= \bigl( (0,0),[\gamma_{1}:1]\bigr)$, where $\gamma_{1}=\frac{\alpha_0}{\alpha_1}$.
There exists a holomorphic function $h_1$ defined on $\Omega$ such that 
$$f_{m_0}=\gamma_{1}g_{m_0}+h_1g_{m_0}$$
 and  that ${h_1}_{|H_{j_0}^{(m_0)}}\equiv 0$.
\end{lemma}
\begin{proof}
 For $m=m_{0}$, and by using the chart $\chi_{m_0+1}$ defined at \eqref{eq:localchart}, we have the presentation
 $F_{m_{0}+1}=(f_{m_{0}+1},g_{m_{0}+1})$, and 
 the equality \eqref{eq:F_{m+1}}  can be written on a neighbourhood $U$ of $H_{j_0}^{(m_0)}\cap\Omega$  as:
 $$(F_{m_{0}+1}\circ \pi_{m_{0}+1}^{-1})(x)=\left(\frac{f_{m_{0}}(x)-\gamma_1g_{m_{0}}}{g_{m_{0}}},g_{m_{0}}\right),$$
 where :
\begin{equation}\label{eq:defgf}
 f_{m_{0}+1}\circ \pi_{m_0+1}^{-1} =\frac{f_{m_{0}} -\gamma_1g_{m_{0}}}{g_{m_{0}}}\ \ \mbox{ and } \ \ 
g_{m_{0}+1}\circ \pi_{m_0+1}^{-1}=g_{m_{0}}
\end{equation}
are holomorphic functions on $U$. 

Moreover, the function $h_{1} :=\frac{f_{m_{0}} -\gamma_1g_{m_{0}}}{g_{m_{0}}}$ is holomorphic on the entire $\Omega$, since $g_{m_{0}}$ has no zero  on  $\Omega \setminus H_{j_0}^{(m_0)}$.
We then  have the identity $f_{m_0}-\gamma_{1}g_{m_0}= h_1g_{m_0}$. We also get ${h_1}_{|H_{j_0}^{(m_0)}}\equiv 0$ since 
 $f_{m_0+1}\equiv 0$ on $H_{j_0}^{(m_0+1)}$.  
\end{proof} 

The identity $g_{m_{0}+1}\circ \pi_{m_0+1}^{-1}=g_{m_{0}}$ implies that  $\ord_{H_{j_0}^{(m_0)}}(g_{m_0})=\ord_{H_{j_0}^{(m_0+1)}}(g_{m_0+1})$. 
The definition of  $m_0$ in the beginning of \S \ref{ss:orderstab} tells that
$\ord_{H_{j_0}^{(m_0+1)}}(F_{m_0+1})=\ord_{H_{j_0}^{(m_0)}}(F_{m_0})$ and therefore
$\ord_{H_{j_0}^{(m_0+1)}}(F_{m_0+1})=\ord_{H_{j_0}^{(m_0+1)}}(g_{m_0+1})$.

For the sets  $\cU_{m}$ defined in Lemma \ref{l:biholomorphism-2}, by their construction we have:
$$\pi_{m_0+1}^{-1}(\cU_{m_{0}}) = \cU_{m_{0}+1},$$
and since 
$$\ord_{H_{j_0}^{(m_0+1)}}(g_{m_0+1})\leq \ord_{H_{j_0}^{(m_0+1)}}(f_{m_0+1}),$$
we get the following upper lever version of Lemma \ref{l:biholomorphism-2}:
\medskip

\noindent
\textbf{[Lemma \ref{l:biholomorphism-2}]$_{m_0+1}$}   \ \ {\it
The restriction
${\pi_{m_0+2}}_{\vert} : \pi_{m_0+2}^{-1}(\cU_{m_{0}+1}) \to \cU_{m_{0}+1}$
is a biholomorphism.}

\medskip
\noindent
which implies that  the restriction:
$${\pi_{m_0+2}}_{\vert} : (\pi_{m_0+1}\circ\pi_{m_0+2})^{-1}(\Omega) \to \Omega$$
 is a biholomorphism.

\bigskip
By using the local chart $\chi_{m_0+2}$ (defined at \eqref{eq:localchart})  at some undetermined point $\xi_{m_0+2}\in S_{m_0+2}$,  like we did before for $S_{m_0+1}$, we may write $F_{m_0+2}=(f_{m_0+2},g_{m_0+2})$, where:
$$f_{m_0+2}=\frac{f_{m_0+1}\circ \pi_{m_0+2}-\gamma_2 (g_{m_0+1}\circ \pi_{m_0+2})}{g_{m_0+1}\circ \pi_{m_0+2}},\ \ 
g_{m_0+2}=g_{m_0+1}\circ \pi_{m_0+2}.$$
for some $\gamma_2\in \bC$ defined by the choice of $\xi_{m_0+2}$,  and  where $f_{m_0+2} \vert H_{j_0}^{(m_0+2)}\equiv 0$ by \eqref{eq:stab}.

By using the presentation \eqref{eq:defgf}  for $f_{m_0+1}$,  we get:   

\begin{equation}\label{eq:rel}
 \left\{ \begin{split}
&f_{m_0+2}=\frac{f_{m_0}- \gamma_{1} g_{m_0}-\gamma_2 g_{m_0}^2}{g_{m_0}^2}\circ \pi_{m_0+1}\circ \pi_{m_0+2}
\\
&g_{m_0+2}=g_{m_0}\circ \pi_{m_0+1}\circ \pi_{m_0+2}
\end{split}  \right.
\end{equation}

The identities \eqref{eq:rel} hold on a neighbourhood of $(\pi_{m_0+1}\circ \pi_{m_0+2})^{-1}\bigl(H_{j_0}^{(m_0)}\cap\Omega \bigr)$ and, as explained in the step before, the function $\frac{f_{m_0}- \gamma_{1} g_{m_0}-\gamma_2 g_{m_0}^2}{g_{m_0}^2}$ is  holomorphic  on the entire set $\Omega$.

Setting $h_2:= \frac{f_{m_0}- \gamma_{1} g_{m_0}-\gamma_2 g_{m_0}^2}{g_{m_0}^2}$, the above relations \eqref{eq:rel} imply the following upgraded version of Lemma \ref{l:level1}:

\medskip
\noindent 
\textbf{[Lemma \ref{l:level1}]$_{2}$}   \ \ \ $\exists\ \gamma_1, \gamma_2\in\bC$ and there is a holomorphic function $h_2: \Omega \to \bC$   such that  $h_2 \vert H_{j_0}^{(m_0)}\equiv 0$  and that
$f_{m_0}=\gamma_1g_{m_0}+\gamma_2g_{m_0}^2+h_2g_{m_0}^2$.  \fin

\medskip

By induction, we thus get the higher upgraded version of Lemma \ref{l:level1}, for any $s\ge 1$:

\medskip
\noindent 
\textbf{[Lemma \ref{l:level1}]$_{s}$}   \ \ \ $\exists\ \gamma_1,\dots,\gamma_s\in\bC$ and there is a holomorphic function $h_s: \Omega \to \bC$ such that  $h_2 \vert H_{j_0}^{(m_0)}\equiv 0$ and that
$f_{m_0}=\gamma_1g_{m_0}+\gamma_2g_{m_0}^2+\cdots+ \gamma_sg_{m_0}^s + h_sg_{m_0}^s$.   \fin

\subsection{Concluding the proof}
\begin{lemma}\label{l:serie}
Let $M$ be a connected complex manifold, let $a\in M$,  and let $f, g:(M,a)\to (\bC,0)$ be holomorphic function germs. We assume that there exists a sequence of complex numbers $\{\gamma_s\}_{s\geq 1}$ and a sequence of holomorphic functions
$\{h_s\}_{s\geq 1}$ defined on an open neighbourhood $a\in \Omega \subset M$ , $h_s: \Omega \to \bC$, such that  $h_s(0) = 0$, and $f=\gamma_1g+\gamma_2g^2+\cdots+\gamma_sg^s+h_sg^s$. 

Then $\jac(f,g)<2$ on $M$.
\end{lemma}
\begin{proof}

We will show that under our assumptions we have
$f=\theta\circ g$,  where $\theta:(\bC, 0)\to (\bC,0)$  is a holomorphic function germ.

\smallskip

\noindent \underline{Step 1.} Let us show  that the formal power series in one variable $\sum_{s=1}^\infty \gamma_s\lambda^s$,  has a positive radius of convergence.

Let $\Lambda: (\bC,0)\to (M,a)$ be a holomorphic function germ such that $g\circ\Lambda\not\equiv 0$. If $\ord_{0}(g\circ\Lambda) =p$ for some  $p\geq 1$,  then, modulo a biholomorphic change of coordinates at $0\in \bC$, one may assume that $g\circ\Lambda(\lambda)=\lambda^p$.   

With these conventions, one writes, for any $s\ge 1$:
\[\hat f(\lambda)=\gamma_1\lambda^{p}+ \gamma_2\lambda^{2p} +\cdots + \gamma_s\lambda^{sp}+\hat h_s(\lambda)\lambda^{sp}\]
where $\hat f:=f \circ\Lambda$,  and $\hat h_s := h_s\circ\Lambda$ with $\hat h_s(0)=0$.  

Let $D'\subset \bC$ be a neighbourhood of the origin such that all the involved functions are holomorphic on $D'$, where by  ``all the involved functions'' we mean precisely $\Lambda, \hat f, g\circ \Lambda$ and $\hat h_s$ for all $s\ge 1$). Let $\overline D(0,r)\subset D'$ denote a disk of radius $r>0$ centred at the origin, and set
$\mu:=\max \bigl(1, \sup\{|\hat f(\lambda)| \  ; \lambda\in \overline D(0,r)\}\bigr)$.

Since $\gamma_j=\frac 1{(jp)!}\hat f^{(jp)}(0)$ for any $j>0$, the
Cauchy inequalities imply that
$|\gamma_j|\leq\frac 1{r^{jp}}\mu$.  Then we obtain $|\gamma_j|^{1/j}\leq\frac 1{r^{p}}\mu^{1/j}\leq\frac 1{r^{p}}\mu$, which shows that the series  $\sum_{s=1}^\infty \gamma_s\lambda^s$  has a positive radius of convergence $\rho = 1/\limsup_{j\ge 1} (|\gamma_j|^{1/j})$.

\medskip

\noindent \underline{Step 2.}  By Step 1 we have proved  the existence of  some neighbourhood $D\subset \bC$ of the origin where $\theta(\lambda):=\sum_{i=0}^\infty \gamma_i\lambda^{j}$ is a holomorphic function. 

It follows that $\theta\circ g$ is a holomorphic function on $g^{-1}(D)$.  By our hypotheses, for any $s\ge 1$, the difference
$$f- \theta\circ g = h_sg^s-\sum_{l\geq s+1}\gamma_lg^l$$
 is then a holomorphic function of order $\ge s$. Since the order of this difference is infinite,  it is identically zero on $g^{-1}(D)$. Consequently, one must have  the identity $f=\theta\circ g$ on $g^{-1}(D)$.
And now, since $f=\theta\circ g$, we have  
 $\jac (f,g) =\jac (\theta\circ g, g) <2$ on $g^{-1}(D)$. Since $M$ is connected, we get $\jac (f,g) <2$ on $M$.
\end{proof}

\smallskip

We apply Lemma \ref{l:serie} to the holomorphic functions $f_{m_0},g_{m_0}$ on $M:= \Reg X_{m_0}$;  we deduce that $\jac (f_{m_0},g_{m_0})<2$ on some open subset of $\Reg X_{m_0}$.
Since  $\pi_1\circ\cdots\circ\pi_m$ is a biholomorphism on the preimage of $X\setminus F^{-1}(0)$, we obtain 
$\jac (f,g)<2$ on some open subset of  $\Reg X$. Now since $\Reg X$ is connected, due to the identity theorem we get that
$\jac (f,g)<2$ on  $\Reg X$. Then Proposition \ref{p:curve} implies  that the image of $F$ is a curve germ,  which contradicts our original assumption that $F$ does not have a well-defined image as a set germ.  

This ends the proof of the finiteness of the algorithm.\fin

\subsection{Proof of Theorem \ref{t:main}}\label{ss:proofmain}
The algorithm at \S \ref{ss:algo} shows how one associates a blossom tree $\cB_{F}$ with thick and thin branches  to a map germ $F:(X, \mathfrak{p})\to (\bC^2, 0)$ which does not have a well-defined image as  a set germ. The finiteness of this tree is proved all along the above Section \ref{s:finiteness}. 
\fin
 
 As a supplement to Theorem  \ref{t:main},  there is the following fact concerning our map germ $F:(X, \mathfrak{p})\to (\bC^2, 0)$, which follows from the above proof. Let $\tau_{1}, \ldots , \tau_{s}$ be the finite sequence of blow-ups needed in the algorithm until there are no more undetermined points, see also the notation used in \S\ref{ss:blossomtree}.
  Let $\tau := \tau_{s}\circ \cdots\circ \tau_{1}: S \to \bC^{2}$ denote the composed map, and let $\bE$ be its total exceptional divisor.
Then:
\begin{corollary}\label{c:im}
 A curve germ $(K,0)$ is in the image $F_*(\cA_{X,\mathfrak{p}})$ if and only if there is a point $q \in \bE$ which belongs to a \emph{thick branch} of the blossom tree  $\cB_{F}$,  and a curve germ $(\widetilde K, q) \subset (S, q)$, such that  $\tau_*(\widetilde K, q) = (K,0)$.
 \fin
\end{corollary}

\subsection{Blossom tree and the image complexity degree}\label{ss:deg}

From Definition \ref{d:degree},  it follows that $\kappa(F)$ is equal to $0$ if and only if $F_*(\cA_{X,\mathfrak{p}}) = \cA_{Y,y}$, and thus, by Remark \ref{r:huck}, if and only if  $F$ is a locally open map.  On the other hand,  in this later case our Algorithm \ref{ss:algo} does not start, and therefore the number of steps is 0.
 
  In case the image of $F$ is a curve germ $(C,0)$, one may construct curve germs $(C_{m},0) \not= (C,0)$ such that $\tan(C_{m},C) =m$ for any $m \gg 1$, and therefore one has $\kappa(F) =\infty$.  
  
  For all other situations, we have the following theorem.

\begin{theorem} \label{t:degree}
 If $F :(X,\mathfrak{p}) \to (\bC^{2},0)$ is a non-constant holomorphic map germ the image of which is not well-defined as a set-germ, then $\kappa(F)$ is precisely the length of the maximal chain of branches in the blossom tree $\cB_{F}$.
 \end{theorem}
\begin{proof}
The minimal number $s\ge 1$ of simultaneous blow-ups which compose the final map $\tau = \tau_{s}\circ \cdots\circ \tau_{1}$ of our above Algorithm \ref{ss:algo} is equal, by definition,  to the length of the longest (maximal) chain of branches in the blossom tree $\cB_{F}$.  Thus what we have to prove here is that the minimal number $s$ of blow-ups  is equal to $\kappa(F)$.

Let 
$C \in F_*(\cA_{X,\mathfrak{p}})$. If $D\in \cA_{\bC^2,0}$ is such that $\tan(C,D)\geq s$ then 
$\tau^*(C)\cap \tau^{-1}(0)=\tau^*(D)\cap \tau^{-1}(0)$. Since $s$ is the last step of the algorithm,
every point of $\tau^{-1}(0)$ must be determined, either full or empty. Thus,  since $C \in F_*(\cA_{X,\mathfrak{p}})$, it follows that 
$\tau^*(C)\cap \tau^{-1}(0)$ is determined-full, which implies that $D \in F_*(\cA_{X,\mathfrak{p}})$. It then follows that $\kappa_C(F)\leq s$, and therefore $\kappa(F)\leq s$.

Let us now denote $\tilde\tau := \tau_{s-1}\circ \cdots\circ \tau_{1}$. Since the algorithm does not stop after $s-1$ steps, it follows that
$\tilde\tau^{-1}(0)$ contains at least an undetermined point, say $\xi\in \tilde\tau^{-1}(0)$. 
In this case there exist $\tilde C,\tilde D\in\cA_{S_{s-1},\xi}$ (see Diagram \eqref{diagram2} for the notation) such that  $C:=\tilde\tau_{*}(\tilde C)\in F_*(\cA_{X,\mathfrak{p}})$ and 
$D:=\tilde\tau_{*}(\tilde D)\not\in F_*(\cA_{X,\mathfrak{p}})$. On the other hand we have $\tan(C,D)\geq s-1$, and this implies the strict inequality $\kappa_C(F)>s-1$. We thus get
 $\kappa_C(F)\geq s$, which in turn implies $\kappa(F)\geq s$. 
 
 This ends our proof  of the equality $s = \kappa(F)$.
\end{proof}

\section{Producing the same blossom tree by a map $(\bC^2, 0)\to (\bC^2, 0)$}\label{s:sametree}

Let $F=(f,g):(X,\mathfrak{p})\to (\bC^2, 0)$ be a holomorphic map germ. We show here that there exists a holomorphic map germ 
$G:(\bC^2, 0)\to (\bC^2, 0)$ such that $G_*(\cA_{\bC^2,0})=F_*(\cA_{X,\mathfrak{p}})$. 

By our finiteness proof,  the ``blossom tree algorithm'' applied to $F$ yields  a finite sequence of blow-ups. As before, we denote by  $\tau := \tau_{r}\circ \cdots \circ \tau_{1}: S\to  \bC^2$  the composition of these blow-up maps, where $S$ is the total space,  
and by $\bE = \cup_{j}E_{j} \subset S$ the exceptional divisor of $\tau$. We say that $\bE$ is the \emph{skeleton} of our blossom tree $\cB_{F}$.

\subsection*{Proof of Theorem \ref{t:main2}.}

In case the image of $F$ is well-defined as a set germ, then by Proposition \ref{p:arkivprop}, either $F$ is locally open, in which case one can take $G$ as the identity, or the image is an irreducible plane curve germ $(C,0)$, and then it is the image of a map $(f,g)$ deduced from a Puiseux expansion of $C$.

We will thus treat the situation when the image of $F$ is not a set germ at $0\in \bC^{2}$.

\smallskip
\noindent
\underline{Case 1.} The tree $\cB_{F}$ has only thin branches. 

At each blow-up there exists  precisely one determined-full point along the exceptional divisor (all the others being determined-empty).  This holds in particular for the last blow-up $\tau : S \to S_{k-1}$ with its exceptional divisor $E_{k}$. Let $\xi\in E_{k}$ be the corresponding determined-full point. 

Since $S$ is nonsingular, let $\iota:(\bC^2,0)\to (S,\xi)$ be a local biholomorphism. Then $G :=\tau\circ \iota$ satisfies our claim.

\smallskip
\noindent
\underline{Case 2.} The tree $\cB_{F}$ has at least one thick branch. 

If $B$ denotes the union of all the components of $\bE$ which correspond to thick branches, then $B$ is a 1-dimensional complex subspace of  $S$, and it is a compact connected  set. (The connectivity of $B$ follows by the same argument that was used in the proof of Lemma \ref{l:onepoint}.)

Our exceptional divisor  $\bE = \cup_{j=1}^{R}E_{j} \subset S$
satisfies the criterion by Grauert \cite{Gr} and Mumford \cite{Mu}, namely  the intersection matrix $[(E_i\cdot E_j)]_{i,j=1,\ldots , k}$ is negative definite. It follows that the subset $B\subset E$ of exceptional divisors also has a negative definite matrix, and therefore by the the same criterion, $B$ is exceptional. This means that there are a two-dimensional complex space $Y$, and
a holomorphic map $\rho:S\to Y$ which contracts $B$ to a point $y\in Y$,  such that the restriction
$S\setminus B\to Y\setminus\{y\}$ is a biholomorphism.

Since $S$ is smooth and hence normal, it follows that $Y$ is normal. We refer to \cite{Gr} for the definitions and results concerning exceptional sets that we use here.

The surface $Y$ contains the image by $\rho$ of the difference 
$\overline{E\setminus B}$. This image consists of those components of $\bE$ corresponding to thin branches.
 
 We define the map $\nu:(Y,y)\to (\bC^2,0)$ such that $\nu(\rho(\overline{E\setminus B}   ) )=0$, and that $\nu_{|Y\setminus\{y\}}=\tau\circ\rho^{-1}$. 
It then follows that $\nu\circ \rho=\tau$, and that $\nu$ is continuous on $Y$ and holomorphic on  $Y\setminus\{y\}$. Since $Y$ is normal, we deduce that $\nu$ is holomorphic on $Y$. 
 
\smallskip

Next, let us set the notation:
 
\begin{equation}\label{eq:SB}
\cA_{S,B}:=\bigsqcup_{\xi\in B}\cA_{S,\xi}.
\end{equation}
 
 We claim that the equality $\rho_*(\cA_{S,B})=\cA_{Y,y}$ holds. Indeed, if $(C,y)\subset (Y,y)$ is some non-constant curve germ, then $\rho^{-1}(C)=C' \cup B$,  where $C'$ is the proper transform of $C$ by $\rho$, a non-degenerated curve. If $C'\cap B=\{\xi\}$, then we have $\rho_*(C',\xi)=(C,y)$. Our claim is proved.

Now, since $\nu\circ \rho=\tau$ and $\rho_*(\cA_{S,B})=\cA_{Y,y}$, we get $\nu_*(\cA_{Y,y})=F_*(\cA_{X,\mathfrak{p}})$.  
It then remains to show that $\cA_{Y,y}$ can be obtained as the image 
of a locally open map germ with target $\bC^{2}$.
This is provided by the following result.
\begin{theorem}\cite{CJ}\label{t:cj}
If $(Y,y)$ is the germ of an irreducible two-dimensional complex space, then there exists a locally open map $\omega:(\bC^2,0)\to (Y,y)$ if and only if there exists a resolution of $(Y,y)$ such that all irreducible components of its exceptional divisor are rational.
\fin
\end{theorem}

In our setting the map  $\rho:S\to Y$ is a resolution of the surface germ $(Y,y)$ and therefore  Theorem \ref{t:cj} aplies. We deduce that there exists a locally open map $\omega:(\bC^2,0)\to (Y,y)$, thus $G=\nu\circ\omega$ satisfies our claim.
This ends the proof of Theorem \ref{t:main2}.


\section{Extension to a singular target}\label{s:extension}

 Let us first go back to Proposition \ref{p:arkivprop} which was  stated for a  map germ $F:(X,\mathfrak{p})\to (\bC^{2},0)$. Its proof (based on the proof of \cite[Proposition 2.1]{JT-arkiv}), actually holds without modifications when we replace the target $(\bC^{2},0)$ by an irreducible surface germ $(Y,y)$. Namely, we have the following statement: \emph{if
the image of $F$ is a well defined set germ at the origin,  then either $(\im F, y)=(Y,y)$ and therefore $F_*(\cA_{X,\mathfrak{p}}) =\cA_{Y,y}$, or $(\im F,y)$  is an irreducible  complex curve germ, and in this case $F_*(\cA_{X,\mathfrak{p}})$ is a single element of  $\cA_{Y,y}$.}  

\smallskip
 Let us now assume that the image of $F:(X,\mathfrak{p})\to (Y,y)$ is not a well-defined set germ.
We will extend  our previous result on $F_*(\cA_{X,\mathfrak{p}})$.

As before, let $(X,\mathfrak{p})$ be an irreducible $n$-dimensional complex analytic  space germ, $n \geq 2$,  and 
 let $(Y,y)$ be an irreducible germ of a singular surface. 
 
Let $\mu_X:\hat X \to X$ and $\mu_Y:\hat Y\to Y$ be their normalizations. 
Since the image of $F$ is not included a the germ of a curve and $(X,\mathfrak{p})$ is an irreducible germ, it follows that
the pre-image by $F$ of the non-normal locus of $Y$ 
is a nowhere dense analytic subset of $X$.  

Then \cite[Proposition,  subsect. 3, sect. 4, ch. 8]{GR} implies that there is a holomorphic map $\hat F:\hat X\to\hat Y$
such that the following diagram is commutative:

\begin{equation*}
\xymatrix{
    \hat X  \ar[d]_{\mu_X} \ar[rr]^{\hat F}    &    &    \hat Y\ar[d]_{\mu_Y}  \\
         X \ar[rr]^F   &   &  Y
} 
\end{equation*}

Since the normalization maps are locally open, we have the equality  $F_*(\cA_{X,\mathfrak{p}})=(\mu_Y)_*(\hat F_*(\cA_{\hat X,\hat{ \mathfrak{p}} }))$,
where $\hat{\mathfrak{p}}=\mu_X^{-1}(\mathfrak{p})$, thus $F_*(\cA_{\hat X,\hat{\mathfrak{p}}})$ determines $F_*(\cA_{X,\mathfrak{p}})$. 

Therefore we may, and will suppose in the following that $Y$ is normal.

\bigskip

Let then $\tau_0:\widetilde Y\to Y$ be a desingularization, and 
let $L=\tau_0^{-1}(y)$ denote the ``exceptional set'' of $\tau_{0}$. Then $L$ is compact by definition, and  connected since $Y$ is irreducible. It may have several irreducible components\footnote{May be assumed nonsingular by passing to a good resolution, but we do not need this assumption here.}. 

We consider the fibered product  $W:= X\times_{Y} \widetilde Y$ together with the two projections $\pr_1:W\to X$, and  
$\pr_2:W\to \widetilde Y$. Let $X_0$ be the closure in $W$ of $\pr_1^{-1}\bigl(X\setminus F^{-1}(y)\bigr)$.
We denote  $F_0 := (\pr_2)_{|X_0}$,  and $\pi_0:= (\pr_1)_{|X_0}$. 
We then  have the following commutative diagram:
\begin{equation}\label{eq:surface}
\xymatrix{
    X_0  \ar[d]_{\pi_0} \ar[rr]^{F_0}    &    &    \widetilde Y\ar[d]_{\tau_0}  \\
         X \ar[rr]^{F}   &   &  Y
} 
\end{equation}

The intersection  
$X_0\cap(\{\mathfrak{p} \}\times L) = \pi_0^{-1}(\mathfrak{p} )$ is a closed analytic subset of  
$ \{\mathfrak{p} \}\times L$.  We have that $X_0\cap(\{\mathfrak{p} \}\times L)$
is connected and thus it is either a point or a connected union of irreducible components of $\{\mathfrak{p} \}\times L$.

Indeed, this connectivity claim is similar to what is shown in the proof of  
 Lemma \ref{l:onepoint}. More explicitly, supposing that $X_0\cap(\{\mathfrak{p} \}\times L)$  is not connected, let $K_1,\dots,K_m$,  be its connected components, where $m\geq 2$.
Let then $V_1,\dots,V_m\subset X_0$ be pairwise disjoint open neighbourhoods of $K_1,\dots,K_m$, respectively.
Let $V$ be a connected open neighbourhood of $\mathfrak{p}$ such that $\pi_0^{-1}(V)\subset \bigsqcup_{j=1}^mV_j$.
The existence of $V$ follows from the properness of $\pi_0$.  By construction, for any $j=1,\ldots ,m$ the set $\pi_0^{-1}(V)\cap V_j$ is an open neighbourhood of $K_j$. 
Since
$\pi_0^{-1}( F^{-1}(y))$ has codimension $\ge 1$ in $X_{0}$, the intersection of open sets
$\pi_0^{-1}(V\setminus F^{-1}(y))\cap V_j$ is non-empty.
Moreover, as $(X,\mathfrak{p})$ is irreducible, the complement 
$V\setminus F^{-1}(y)$ is connected, and since the restriction 
${\pi_0}_{|}: \pi_0^{-1}(V\setminus F^{-1}(y)) \to V\setminus F^{-1}(y)$ is a bi-holomorphism, it follows that $\pi_0^{-1}(V\setminus F^{-1}(y))$ is also connected, and this contradicts the fact that it is included in the disjoint union  $\bigsqcup_{j=1}^mV_j$ of more than two open sets. Our claim is proved.

\

We continue our construction.
An irreducible component $L_1$ of $L$ is called \emph{thick} if  
 the inclusion $\{\mathfrak{p} \}\times L_1\subset \pi_0^{-1}(\mathfrak{p})$ holds; it is called \emph{thin} if this inclusion does not hold.   As before,  a point $\xi\in L$ is called:
 
 \medskip
 
 $\bullet$ {determined-empty}, if $\xi\not\in F_0(\pi_0^{-1}(\mathfrak{p} ))$.
 
 $\bullet$ {determined-full}, if $\xi\in F_0(\pi_0^{-1}(\mathfrak{p} ))$ and
${F_0}_*\bigl(\cA_{X_0,(\mathfrak{p} ,\xi)}\bigr) = \cA_{\widetilde Y,\xi}$.

 $\bullet$ {undetermined}, if  $\xi\in F_0(\pi_0^{-1}(\mathfrak{p} ))$ and
${F_0}_*\bigl(\cA_{X_0,(\mathfrak{p} ,\xi)}\bigr) \subsetneq \cA_{\widetilde Y,\xi}$.

 \medskip

By the same arguments as those used in the proof of Lemma \ref{l:finite}, applied to diagram \eqref{eq:surface},  one  shows that the set of undetermined points is finite.
Let then $\{\xi_1,\dots,\xi_r\}\subset L$ be the set of these undetermined  points.

At this level we start to apply our Algorithm \ref{ss:algo} to each map germ $F_0:(X_0,(\mathfrak{p} ,\xi_j))\to (\widetilde Y,\xi_j)$, where $j$ runs from $1$ to $r$. The result is a finite number of blossom trees $\cB_{F_{0}, j}$, one for each point $\xi_{j} \in \{\xi_1,\dots,\xi_r\}$. We
therefore have the exceptional divisor $L$ of $\pi_{0}$, with thick and thin components, on which grow $r$ blossom trees. This building is a new blossom tree like structure; we will call it the ``forest on $L$'' associated to $F_{0}$, and denote it by $\cB_{F, \pi_{0}}$.

If we apply the Algorithm \ref{ss:algo} to the above named $r$ points of $L$ in a  simultaneous manner, than we operate
a finite number of simultaneous blow-ups at each blossom tree; let  $\tau_{j}$ denote the total blow-up map at the step $j$.
We obtain a finite sequence of blow-ups:
\[
 S^{(s)} \stackrel{\tau_{s}}{\longrightarrow} S^{(s-1)} \stackrel{\tau_{s-1}}{\longrightarrow} \cdots \stackrel{\tau_{1}}{\longrightarrow} \widetilde Y
\]
and a commutative diagram:
\begin{equation}
\xymatrix{
    X_s  \ar[d]_{\pi_s} \ar[rr]^{F_s}    &    &    S^{(s)}\ar[d]_{\tau_s}  \\
         X_0 \ar[rr]^{F_0}   &   &  \widetilde Y .
} 
\end{equation}

Setting $S:=S^{(s)}$, and $\tau := \tau_{1}\circ \cdots \circ \tau_{s}$, we have proved the following:

\begin{theorem}\label{t:singsurf}
 Let $(X,\mathfrak{p})$ be an irreducible complex analytic set germ and $F:(X,\mathfrak{p})\to (Y,y)$ be a holomorphic map germ
such that the image of $F$ is not a well-defined set germ. 
Then $F_*(\cA_{X,\mathfrak{p}})=(\tau\circ \tau_0)_*(\cB_{F, \pi_{0}})$.
\fin
\end{theorem}

\begin{remark}
 Let $B\subset S$ denote the union of the thick components of $\cB_{F, \pi_{0}}$. Recalling our notation 
\eqref{eq:SB}:
$\cA_{S,B}:=\bigsqcup_{\xi\in B}\cA_{S,\xi}$, 
we may give the similar remark as that stated as Corollary \ref{c:im}, namely: 

A curve germ $(K,0)$ is in the image $F_*(\cA_{X,\mathfrak{p}})$ if and only if there is a point $q$ on some \emph{thick branch} of the blossom tree  $\cB_{F, \pi_{0}}$,  and a curve germ $(\widetilde K, q) \subset (S, q)$, such that  $(\tau\circ \tau_0)_*(\widetilde K, q) = (K,0)$.
\end{remark}

\medskip
 
\section{examples}\label{s:ex}
\begin{example}\label{ex:zero}
 Let $X:= \bC^{2}$, and $F:(\bC^2,0)\to (\bC^2,0), \ \ F(x,y)=(xy, y).$

We blow-up the target $\bC^2$ at the origin, with $E_1$ as its exceptional divisor.
$$\widetilde{\bC^2}=\{(u,v,[\xi_0:\xi_1]):u\xi_1=v\xi_0\}\subset \bC^2\times\bP^1, \ \ E_1=\{(0,0)\}\times\bP^1.$$

We have $F^{-1}(0)=\{(x,y)\in\bC^2 \mid y=0\}$. By our definition in \S \ref{ss:blowup}: 
$$X_1 = \cl  \bigl\{(x,y),(u,v),[\xi_0 : \xi_1])\in\bC^2\times\bC^2\times\bP^1 \mid u=xy,v=y,u\xi_1=v\xi_0, y\neq 0\bigr\} ,$$
 where the closure is taken in the fibered product space:
\[
\bigl\{ (x,y),(u,v),[\xi_0:\xi_1])\in\bC^2\times\bC^2\times\bP^1 \mid u=xy,v=y,xy\xi_1=y\xi_0\bigr\}.
\]
Dividing out by $y\neq 0$ we get:
$$X_1\subseteq \{(x,y),(u,v),[\xi_0:\xi_1])\in\bC^2\times\bC^2\times\bP^1 \mid u=xy,v=y,x\xi_1=\xi_0\},$$
which is actually an equality since the right hand side is nonsingular and connected.
Next
$$ \pi_1^{-1}(0,0)= X_{1} \cap \{x=0,y=0\} = \{(0,0),(0,0),[0:1]\}. $$
We get $F_1(\pi_1^{-1}(0,0))=((0,0),[0:1])$, which shows that all points in $E_1$
are determined-empty except of the point $[0:1]$. 
This means that $E_1$ is thin, and it remains to find the nature of $[0:1]$, i.e. what happens when we further blow-up at this point.
In the chart 
$$\widetilde{\bC^2}=\{(u,v,[\xi_0:\xi_1]) \mid u\xi_1=v\xi_0,\ \xi_1\neq 0\} \to \bC^2$$
we set coordinates $(\lambda,v)\in \bC^2$, where $\lambda=\frac{\xi_0}{\xi_1}$, with
 inverse map $(\lambda,v)\mapsto(v\lambda,v, [\lambda:1])$.

In these coordinates $X_1$ presents as 
$$X_1=\{(x,y),(u,v),\lambda)\in\bC^2\times\bC^2\times\bC \mid u=xy,v=y,x=\lambda\},$$
has independent coordinates $(\lambda,y)$, and thus it is isomorphic to $\bC^2$. In these coordinates,
 $F_1$ (which is the restriction to $X_{1}$ of the projection $X \times \widetilde{\bC^2}\to \widetilde{\bC^2}$) reads as the identity $F_1(\lambda,y)=(\lambda,y)$. This shows that $[0:1]$ is determined-full.  The dual graph of the blossom tree $\cB_{F}$ (cf \S\ref{ss:blossomtree}) is then a dotted disk:
 
 \medskip
 
   $$\odot^{1}$$
\end{example}

\medskip

\begin{example}\label{ex:one} 
Let $X:= \bC^{2}$, and $F:(\bC^2,0)\to(\bC^2,0)$ \ \ $F(x,y)=(x^3y^3,x^2y)$.

By blowing-up the target $\bC^2$ at 0, denoting by $E_1$ the exceptional divisor, we get
that $X_1$ is the closure of the intersection:
$$\bigl\{ ((x,y),(u,v),[\xi_0:\xi_1])\in\bC^2\times\bC^2\times\bP^1 \mid   u=x^3y^3,v=x^2y,u\xi_1=v\xi_0\bigr\} \cap \{ xy\neq 0\}.$$
in the first set. We obtain the inclusion: 
\begin{align*}
X_1\subseteq
\bigl\{ ((x,y),(u,v),[\xi_0:\xi_1])\in\bC^2\times\bC^2\times\bP^1 \mid   u=x^3y^3,v=x^2y,xy^2\xi_1=\xi_0\bigr\},
\end{align*}
which is actually an equality since the right hand side is non-singular and connected, hence irreducible.

\sloppy Since $\pi_1$ is the restriction to $X_1$ of the projection to $(x,y)$, we get
$\pi_1^{-1}(0,0)=\{(0,0),(0,0),[0:1]\}$. This shows that $E_1$ is thin since all points of  $E_1$ are determined-empty  (in the terminology of Definition \ref{d:points}) with the exception of $[0:1]$.  One has to study the situation at this point, as follows.

We work in the chart  $\xi_1\neq 0$ of $\bP^1$, with the coordinate function 
$\lambda= \xi_0/\xi_1$.  Then $X_1\cap\{\xi_1\neq 0\}=\{((x,y,u,v,\lambda)\in\bC^5 \mid   u=x^3y^3,v=x^2y,xy^2=\lambda\}$ which is isomorphic to $\bC^2$ of coordinates $(x,y)$.

Recalling that $F_1$ is the restriction to $X_1$ of  the projection to the coordinates $((u,v),[\xi_0:\xi_1])$, by using the above charts and isomorphism, $F_1$ presents as  $F_1(x,y)=(xy^2,x^2y).$
As a map germ $F_{1}: (\bC^{2}, 0) \to (\bC^{2}, 0)$,  this is not open at the origin of the target. We deduce that our point $\pi_1^{-1}(0,0) \in E_{1}$ is undetermined.

 We continue the algorithm by blowing-up the target of $F_1$ at this point. Let
$E_2=\{(0,0)\}\times\bP^1\subset \bigl\{(\hat u,\hat v,[\hat \xi_0:\hat\xi_1] \mid \hat u\hat\xi_1=\hat v\hat\xi_0\bigr\}$ be the exceptional divisor of this new blow-up. 
We have
$F_1^{-1}(0,0)=\{(x,y)\in\bC^2 \mid xy=0\}$, hence $X_2$ is the closure of the intersection:
$$\bigl\{((x,y),(\hat u,\hat v),[\xi_0:\xi_1])\in\bC^2\times\bC^2\times\bP^1 \mid   \hat u=xy^2,\hat v=x^2y,\hat u\hat \xi_1=\hat v\hat\xi_0\bigr\}\cap\{xy\neq 0\}$$
in the first set.

By a couple of standard operations, we get the inclusion:
\begin{align*}
X_2\subseteq
\bigl\{((x,y),(\hat u,\hat v),[\hat \xi_0:\hat \xi_1])\in\bC^2\times\bC^2\times\bP^1 \mid   \hat u=xy^2,\hat v=x^2y,y\hat \xi_1=x\hat\xi_0\bigr\}
\end{align*}
which once again turns out to be an equality, since the right hand side is a connected nonsingular variety, thus irreducible.

We have 
$\pi_2^{-1}(0,0)=\{(0,0)\}\times\{(0,0)\}\times\bP^1$, thus $E_2$ is thick.  Let us show that the point $[a:1]\in E_2$ is determined-full for $a\neq 0$ and undetermined if $a=0$.

In the chart $\hat\xi_1\neq 0$ of $\bP^1$, with  $\hat\lambda := \hat \xi_0 / \hat \xi_1$, we get:
$$X_2\cap \{\hat\xi_1\neq 0\}= \bigl\{(x,y,\hat u,\hat v,\hat \lambda)\in\bC^5 \mid   \hat u=xy^2,\hat v=x^2y,y=x\hat\lambda \bigr\}
$$
which is isomorphic to $\bC^2$ with coordinates $(x,\hat \lambda)$. We also consider the isomorphism of
$\widetilde{\bC^2}\cap\{\hat\xi_1\neq 0\}=\{(\hat u,\hat v,[\hat\xi_0,\hat\xi_1] \mid  \hat u\hat\xi_1=\hat v\hat\xi_2, \hat\xi_1\neq 0\} \simeq  \bC^2$ with $\bC^2$ in coordinates $(\hat\lambda,\hat v)$.

By using these coordinates, we now obtain the presentation: $F_2(x,\hat \lambda)=(\hat \lambda,x^3\hat \lambda)$.
As a map germ $(\bC^2,(0,a)) \to (\bC^2,(a,0))$, this is open at  $(a,0)$ for any $a\neq 0$,  but not open for $a=0$. This ends the proof that $[0:1]\in E_2$ is the only undetermined point in this chart. 
Pursuing the study in the chart $\hat\xi_1\neq 0$, we blow-up at $\xi = [0:1]$. By similar computations, we find 
$F_{3} : (\bC^2,(0,0)) \to (\bC^2,(0,0))$, $F_3(x,\hat \lambda)=(\hat \lambda,x^3)$, which is an open map germ.  This shows that the exceptional divisor $E_{3}$ is thin and it is a terminal branch.

Similar computations done in parallel in the other chart $\hat\xi_0\neq 0$ show that $[1:0]\in E_2$ is the only undetermined point, and that after a new blow-up at this point we obtain a new exceptional divisor $E'_{3}$ which is thin and terminal.

We get the following picture of the dual graph of the blossom tree $\cB_{F}$:

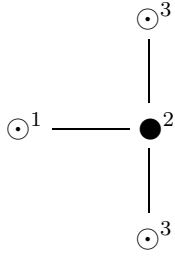
\medskip

\xymatrix{
&&&&&&\ \odot^3\\
&&&&&\odot^{1}\ar@{-}[r]& \CIRCLE^2\!\!\!\ar@{-}[u]\ar@{-}[d]\\
&&&&&&\ \odot^3
}

\end{example}

\medskip

 \begin{example}\label{ex:n+1,1}
Let $X := \{(x,y,z)\in\bC^3 \mid x^{n+1}=yz\}$,  $F: (X, 0) \to (\bC^2,0), \ \ F(x,y,z)=(x,y).$

Testing with Proposition \ref{p:arkivprop} we get the image of $F$ is not well-defined as set germ since. Indeed $F$ it is not locally open (since the line $(\bC\times\{0\},0)$ is not in the image), and its image is not a curve (since the criterion of Proposition \ref{p:curve} does not hold).  We then start to compute the blossom tree structure of $F_*(\cA_{X,0})$.

\medskip

We blow-up the target $\bC^2$ at the origin, with $E_1$ as the exceptional divisor, and
$F^{-1}(0,0)=\{0\}\times\{0\}\times\bC$. Hence $X_1$ is the closure of the following intersection:
$$\bigl\{ ((x,y,z),(u,v),[\xi_0:\xi_1])\in\bC^3\times\bC^2\times\bP^1 \mid  x^{n+1}=yz, u=x,v=y,u\xi_1=v\xi_0\bigr\} \cap \{(x,y)\neq (0,0)\}$$
 in the first set.

By using standard algebraic manipulations and by taking $y\neq 0$ (since $y=0$ implies $x=0$) one shows the inclusion:
\begin{align*}
X_1\subseteq &
\{((x,y,z),(u,v),[\xi_0:\xi_1])\in\bC^3\times\bC^2\times\bP^1 \mid \\
&  x^{n+1}=yz, u=x,v=y,u\xi_1=v\xi_0, z\xi_1^{n+1}=y^{n}\xi_0^{n+1}\}
\end{align*}
and one can show that this last algebraic is irreducible, which tells that we actually have an equality. 
What is less easy here is to find the equations which define the above irreducible variety.

We see that $\pi_1^{-1}(0,0,0)$ is isomorphic to $\bP^1$, which means that $E_1$ is \emph{thick}.
  Standard computations show that all points in $\{[\xi_0:\xi_1] \mid \xi_1\neq 0\}$ are determined-full.  In the chart  $\xi_0\neq 0$ of coordinate $\lambda = \xi_{1}/\xi_{0}$, the component $X_1$ presents as:
$$\{((x,y,z),(u,v),\lambda)\in\bC^3\times\bC^2\times\bC \mid  x^{n+1}=yz, u=x,v=y,u\lambda=v, z\lambda^{n+1}=y^{n}\}
$$
which, by projection in $\bC^4$ and reduction, identifies to:
$$
\{((x,y,z,\lambda)\in\bC^4 \mid x^n=z\lambda, x\lambda=y\}.
$$
By another projection, this  identifies in turn to the set $\{(x,z,\lambda) \mid x^n=z\lambda\}\subset \bC^3.$

In the chart coordinates $(u,\lambda)$ for 
$\widetilde{\bC^2}=\{(u,v,[\xi_0:\xi_1]):u\xi_1=v\xi_0\}\subset \bC^2\times\bP^1$, the map
$F_1$ presents as:
$$ \{((x,z,\lambda)\in\bC^3 \mid x^n=z\lambda\}\ni(x,z,\lambda) \mapsto
(x,\lambda)\in\bC^2
$$
which is similar to $F$, namely one replaces the exponent $n+1$ by $n$.

By iterating the above computations, the exponent  drops by $1$ at each step. At the last step we get the presentation: 
$$F_n: (X_{n}, 0)\to (\bC^2,0), \ \ F_n(x,y, z)=(x,y),$$
where $X_{n} = \{(x,y,z) \mid x=yz\}$. Since $X_{n}$ is isomorphic to $\bC^2$ via the projection $(x,y,z)\to (y,z)$,  the map $F_n$ identifies to
$F_n(y,z)=(yz,z)$. By a new blow-up, this  generates one thin divisor with a determined-full point (as in Example \ref{ex:zero}).

The dual graph of the blossom tree $\cB_{F}$ is therefore:

\medskip

\xymatrix{&&&&\CIRCLE^1\ar@{-}[r]&\CIRCLE^2\ar@{-}[r]&\cdots\ar@{-}[r]&\CIRCLE^n\ar@{-}[r]&\odot^{n+1}}

\end{example}

 \begin{example}\label{ex:n+1,n}
Let $X=\{(x,y,z)\in\bC^3\ \vert \ x^{n+1}=y^{n}z\}$ for $n\in\bZ$, $n\geq 1$, and $F:X\to\bC^2$ given by $F(x,y,z)=(x,y)$. Notice first that the image of $F$ is not well defined as a set germ, since (cf  Propositions \ref{p:arkivprop} and \ref{p:curve}) it is neither locally open, nor a curve germ (because $\jac F$ is not less than 2 on $\Reg X$).

 We may then apply our algorithm. Despite the apparent similarity to the preceding example, we obtain a totally different shape for the dual graph of the blossom tree $\cB_{F}$:

\medskip

\xymatrix{
&&&&& \CIRCLE^1\ar@{-}[r]& \odot^{2}
}

\end{example}

\bigskip


\end{document}